\documentclass[a4j,12pt]{article}  
\usepackage{geometry}
\usepackage[T1]{fontenc} 
\usepackage{times} 
\usepackage{amsmath,amssymb,bm, mathpazo}
\usepackage{amsfonts}
\usepackage{pifont}
\usepackage{mathrsfs}
\usepackage{graphicx}
\usepackage{amscd}

\makeatletter

\newtheorem{thm}{\theoremname}[section]
\newtheorem{meidai}[thm]{\propositionname}
\newtheorem{cor}[thm]{\corollaryname}

\newtheorem{lem}[thm]{\lemmaname}

\newtheorem{prob}{\problemname}
\newtheorem{exa}{\examplename}[section]
\newtheorem{rem}{\remarkname}[section]

\newcommand{\theoremname}{Theorem}
\newcommand{\definitionname}{Definition}
\newcommand{\lemmaname}{Lemma}
\newcommand{\corollaryname}{Corollary}
\newcommand{\axiomname}{Axiom}
\newcommand{\propositionname}{Proposition}
\newcommand{\problemname}{Question}
\newcommand{\examplename}{Example}
\newcommand{\remarkname}{Remark}

\def\tedsymbol{\vcenter{\hbox{\vrule\@height.5em\@width.5em}}}
\def\ted{{\unskip\nobreak\hfil\penalty50
 \quad\hbox{}\nobreak\hfil \hbox{$\tedsymbol$}
 \parfillskip\z@ \finalhyphendemerits\z@\par}}

\makeatother

\title{On certain infinite families of imaginary quadratic fields whose Iwasawa $\lambda$-invariant is equal to $1$}
\author{
\LARGE{Akiko Ito}\footnote{Supported by Grant-in-Aid for JSPS Fellows (22-222) from Japan Society for the Promotion of Science.}}
\begin{document}
\date{\empty}
\maketitle
{\small {\bf Abstract.}
Let $p$ be an odd prime number.
In this paper, we show existence of certain infinite families of imaginary quadratic fields in which $p$ splits and  
whose Iwasawa $\lambda$-invariant of the cyclotomic $\mathbb{Z}_p$-extension is equal to $1$.}
 \section{Introduction} \label{sec1}
Throughout this paper, $D$ will denote the fundamental discriminant of a quadratic field $\mathbb{Q}(\sqrt{D})$.
Let $\chi_D := \bigl(\frac{D}{\cdot}\bigr)$ be the Kronecker character.
For a prime number $p$, 
we denote by $\lambda_p(\mathbb{Q}(\sqrt{D}))$ the Iwasawa $\lambda$-invariant of the cyclotomic $\mathbb{Z}_p$-extension 
of $\mathbb{Q}(\sqrt{D})$.
If $p$ splits in an imaginary quadratic field $\mathbb{Q}(\sqrt{D})$, then it is known that $\lambda_p(\mathbb{Q}(\sqrt{D}))$ is greater than or equal to $1$.
We see how often imaginary quadratic fields with $\lambda_p = 1$ appear for a given prime number $p$.
First, we consider the following question.
\begin{prob} \label{q1}
Let $p$ be a prime number.
Is the set
\begin{eqnarray}
\left\{D:
\begin{array}{c|c}
fundamental \ discriminant \ of & {} \nonumber \\
imaginary \ quadratic \ field & {} \nonumber \\
\end{array}
\lambda_p(\mathbb{Q}(\sqrt{D})) = 1 \ and \ \chi_D(p) = 1 \right \}
\end{eqnarray}
infinite?
\end{prob}
For $p=2$, Question \ref{q1} is solved affirmatively.  
In fact, we can prove this by using Kida's formula 
on the Iwasawa $\lambda$-invariants of the cyclotomic $\mathbb{Z}_2$-extensions of imaginary quadratic fields \cite{Kid79}.

We treat the case where $p \ge 3$.
D.~S.~Dummit, D.~Ford, H.~Kisilevsky and J.~W.~Sands~\cite{DFKS91}, T.~Fukuda and H.~Taya~\cite{FT}, 
J.~S.~Kraft and L.~C.~Washington~\cite{KW07}, etc. constructed tables of the Iwasawa $\lambda$-invariants of the cyclotomic $\mathbb{Z}_p$-extensions of 
imaginary quadratic fields.
For $p = 3$, T.~Fukuda and H.~Taya~\cite[p. 302]{FT} showed the following table of $\lambda_3(\mathbb{Q}(\sqrt{-d}))$
for positive square-free integers $d$'s less than $10,000,000$.\\
~\\
\begin{tabular}{|c|c|c|c|c|c|c|}
\hline
$\lambda_3$ & 0 & 1 & 2 & 3 & 4 & 5 \\
\hline
&&\\[-14pt]\hline          
$d \equiv 0 \bmod 3$ & 890546 \ ($\ast$) & 409063 & 145360 & 49796 & 16750 & 5517 \\
\hline
$d \equiv 1 \bmod 3$ & 1327112 & 617243 & 220648 & 76138 & 25595 & 8666 \\
\hline
$d \equiv 2 \bmod 3$ & 0 & 1320935 & 618333 & 225217 & 76691 & 25554 \\
\hline
\end{tabular}\\
~\\
~\\
\begin{tabular}{|c|c|c|c|c|c|c|c|c|c|c|}
\hline
$\lambda_3$ & 6 & 7 & 8 & 9 & 10 & 11 & 12 & 13 & 14 & Total  \\
\hline
&&\\[-14pt]\hline          
$d \equiv 0 \bmod 3$ & 1864 & 613 & 218 & 50 & 24 & 8 & 1 & 1 & 2 & 1519813 \\
\hline
$d \equiv 1 \bmod 3$ & 2939 & 912 & 329 & 112 & 28 & 12 & 4 & 2 & 1 & 2279741 \\
\hline
$d \equiv 2 \bmod 3$ & 8622 & 2956 & 939 & 311 & 123 & 44 & 6 & 3 & 2 & 2279736 \\
\hline
\end{tabular}\\
~\\
~\\
The number in the case $(\ast)$ of this table denotes the number of positive square-free integers $d$ less than $10,000,000$ 
such that $d \equiv 0 \bmod 3$ and $\lambda_3(\mathbb{Q}(\sqrt{-d})) = 0$.
It seems that the Iwasawa $\lambda$-invariants of the cyclotomic $\mathbb{Z}_3$-extensions 
of imaginary quadratic fields tend to be small.
  
Question \ref{q1} was studied in D.~Byeon~\cite{By05}.
Suppose $0 < X \in \mathbb{R}$. We denote by $S_-(X)$ the set of negative fundamental discriminants $-X < D < 0$ of quadratic fields.
Byeon proved the following theorem.
\begin{thm}[Byeon, {[5, Proposition 1.2]}] \label{thm3}
Let $p$ be an odd prime number.
Assume that there is a negative fundamental discriminant $D_0$ of a quadratic field which satisfies the following conditions:
\begin{equation} 
 \begin{split}
&{\rm (i)} \ \lambda_p(\mathbb{Q}(\sqrt{D_0})) = 1, \nonumber \\
&{\rm (ii)} \ \chi_{D_0}(p) = 1.
 \end{split}
\end{equation}                       
Then, for any sufficiently large $X \in \mathbb{R}$, we have
\begin{eqnarray}
\sharp \left\{
\begin{array}{c|c}
D \in S_-(X) & \lambda_p(\mathbb{Q}(\sqrt{D})) = 1 \ and \ \chi_{D}(p) = 1 \nonumber \\
\end{array}
\right \} \gg \frac{\sqrt{X}}{\log X}.
\end{eqnarray}
\end{thm}
The assumption of the existence of $D_0$ is necessary for the proof.
In \cite{By05}, Byeon gave such $D_0$ for each odd prime number by using a result of R.~Gold~\cite[Theorem 4]{Go74}.
But it seems that he did not use Gold's result correctly
(he does not verify the indivisibility of the class number).
Therefore, we give such $D_0$ in the following way.
\begin{thm} \label{thm4}
Let $p$ be an odd prime greater than $3$.
If $\lambda_p(\mathbb{Q}(\sqrt{1 - p}))$ is greater than $1$, then $\lambda_p(\mathbb{Q}(\sqrt{4 - p}))$ is equal to $1$.
\end{thm}
\begin{exa} \label{exa5}
$(1)$ \ When $p = 13$, we have 
$$\lambda_{13}(\mathbb{Q}(\sqrt{1 - 13})) = \lambda_{13}(\mathbb{Q}(\sqrt{- 3})) = 2 > 1$$
and 
$$\lambda_{13}(\mathbb{Q}(\sqrt{4 - 13})) = \lambda_{13}(\mathbb{Q}(\sqrt{- 1})) = 1.$$
$(2)$ \ When $p = 23$, we have 
$$\lambda_{23}(\mathbb{Q}(\sqrt{1 - 23})) = \lambda_{23}(\mathbb{Q}(\sqrt{- 22})) = 2 > 1$$
and 
$$\lambda_{23}(\mathbb{Q}(\sqrt{4 - 23})) = \lambda_{23}(\mathbb{Q}(\sqrt{- 19})) = 1.$$
\end{exa}
We calculated these examples by using Mizusawa's program \cite{Mi}.
Since the integers $1 - p$ and $4 - p$ are quadratic residues modulo $p$,
we can take $\mathbb{Q}(\sqrt{1 - p})$ or $\mathbb{Q}(\sqrt{4 - p})$ as $\mathbb{Q}(\sqrt{D_0})$ when $p \ge 5$.
When $p = 3$, we can take $\mathbb{Q}(\sqrt{-23})$ as $\mathbb{Q}(\sqrt{D_0})$.
From Theorems \ref{thm3} and \ref{thm4}, we obtain the following corollary.
\begin{cor} \label{cor6}
Let $p$ be an odd prime number.
Then, for any sufficiently large $X \in \mathbb{R}$, we have
\begin{eqnarray}
\sharp \left\{
\begin{array}{c|c}
D \in S_-(X) & \lambda_p(\mathbb{Q}(\sqrt{D})) = 1 \ and \ \chi_{D}(p) = 1\nonumber \\
\end{array}
\right \} \gg \frac{\sqrt{X}}{\log X}.
\end{eqnarray}
\end{cor}
Question \ref{q1} is solved affirmatively.
Secondly, we will study refinement of Question \ref{q1}.
\begin{prob} \label{q7}
Let $p$, $r_1$, $r_2$, ..., $r_s$ and $r'_1$, $r'_2$, ..., $r'_t$ be distinct odd prime numbers, where $s$ and $t$ are positive integers.
Is the set
\begin{eqnarray}
\left\{
\begin{array}{c|c}
{} & \lambda_p(\mathbb{Q}(\sqrt{D})) = 1, \nonumber \\
D: \ fundamental \ discriminant & \chi_D(p) = 1, \nonumber \\
of \ imaginary \ quadratic \ field & \chi_D(r_1) = \cdots = \chi_D(r_s) = 1, \ and \nonumber \\
{} & \chi_D(r'_1) = \cdots = \chi_D(r'_t) = -1 \nonumber \\
\end{array}
\right\}
\end{eqnarray}
infinite?
\end{prob}
We study existence of infinite families of imaginary quadratic fields with $\lambda_p = 1$ under splitting conditions of prime numbers.
The results \cite[Theorem 13]{JO99} and \cite[Theorem]{Kim03} gave a hint on raising this question.
We give a generalization of Theorem \ref{thm3}.
\begin{thm} \label{thm8}
Let $p$ be an odd prime number 
and let $\mathfrak{S}_+$ and $\mathfrak{S}_-$ be mutually disjoint finite sets of odd prime numbers such that $p \in \mathfrak{S}_+$.
Fix $(A, B) \in \{(1, 8), (5, 8), (8, 16)\}$.
Assume that there is a negative fundamental discriminant $D_0$ of a quadratic field which satisfies the following conditions:
\begin{equation} 
 \begin{split}
{\rm (i)} \ &D_0 \equiv A\bmod B, \ \ \ {\rm (ii)} \ D_0 \neq -8, \ \ \ {\rm (iii)} \ \lambda_p(\mathbb{Q}(\sqrt{D_0})) = 1, \nonumber \\
{\rm (iv)} \ &every \ prime \ number \ r \in \mathfrak{S}_+ \ splits \ in \ \mathbb{Q}(\sqrt{D_0}) \ and \ every \ prime \nonumber \\
                &number \ r' \in \mathfrak{S}_- \ is \ inert \ in \ \mathbb{Q}(\sqrt{D_0}).
 \end{split}
\end{equation}                       
Then, for any sufficiently large $X \in \mathbb{R}$, we have
\begin{eqnarray}
\sharp \left\{
\begin{array}{c|c}
{} & D \equiv A\bmod B, \nonumber \\
D \in S_-(X) & \lambda_p(\mathbb{Q}(\sqrt{D})) = 1, \ and \nonumber \\
{} & condition \ (\ast ) \ holds \nonumber \\
\end{array}
\right \} \gg \frac{\sqrt{X}}{\log X}, 
\end{eqnarray}
where $(\ast )$ denotes the condition that every prime number $r \in \mathfrak{S}_+$ splits in $\mathbb{Q}(\sqrt{D})$
 and every prime number $r' \in \mathfrak{S}_-$ is inert in $\mathbb{Q}(\sqrt{D})$.
\end{thm}
\begin{exa} \label{exa9}
Assume $p = 3$, $\mathfrak{S}_+ := \{3\}$, $\mathfrak{S}_- := \{5\}$, and $(A, B) = (1, 8)$. 
In this case, we can take $D_0 = -23$, for example.
It follows from Theorem \ref{thm8} that for any sufficiently large $X \in \mathbb{R}$ we have
\begin{eqnarray}
\sharp \left\{
\begin{array}{c|c}
{} & D \equiv 1\bmod 8, \nonumber \\
D \in S_-(X) & \lambda_3(\mathbb{Q}(\sqrt{D})) = 1, \nonumber \\
{} & \chi_D(3) = 1, \ and \ \chi_D(5) = -1 \nonumber \\
\end{array}
\right \} \gg \frac{\sqrt{X}}{\log X}.
\end{eqnarray}
Since the condition that $D \equiv 1\bmod 8$, $\chi_D(3) = 1$, and $\chi_D(5) = -1$ is equivalent to the congruence relation $D \equiv 73, 97 \bmod 120$,
the above equation implies that
$$\sharp\{D \in S_-(X) \mid D \equiv 73, 97 \bmod 120 \ and \ \lambda_3(\mathbb{Q}(\sqrt{D})) = 1\} \gg \frac{\sqrt{X}}{\log X}.$$
\end{exa}
\begin{rem} \label{rem10}
We give a remark on imaginary quadratic fields whose Iwasawa $\lambda_p$-invariant is equal to $1$.
Let $K$ be an imaginary quadratic field and $p$ an odd prime number such that $p$ splits in $K$.
It is known that if $\lambda_p(K) = 1$, then Generalized Greenberg's Conjecture (called GGC, \cite{Gr73}) holds for $K$ and $p$.
Therefore, if $D_0$ satisfying the above properties exists, 
we get infinite families of imaginary quadratic fields for which GGC holds true.
\end{rem}
Thirdly, as a topic relevant to Question \ref{q1}, we consider the following question.
\begin{prob} \label{q1.1}
Let $p$ be an odd prime number.
Is the set
\begin{eqnarray}
\left\{D:
\begin{array}{c|c}
fundamental \ discriminant \ of & {} \nonumber \\
imaginary \ quadratic \ field & {} \nonumber \\
\end{array}
\lambda_p(\mathbb{Q}(\sqrt{D})) > 1 \ and \ \chi_D(p) = 1 \right \}
\end{eqnarray}
infinite?
\end{prob}
Question \ref{q1.1} is solved affirmatively by J.~W.~Sands~\cite{Sa93}.
The outline of his proof is as follows.
Fix an odd prime number $p$ and an arbitrary integer $n \ge 2$ which is not divisible by $p$.
Define 
$$A_{p, n} := \{a \in \mathbb{Z} \mid 0 < a < 2p^n, \ p \nmid a, \ and \ a^{p-1} \equiv 1 \bmod p^2\}.$$
He proved that $p$ splits in $\mathbb{Q}(\sqrt{a^2 - 4p^{2n}})$ and that
$\lambda_p(\mathbb{Q}(\sqrt{a^2 - 4p^{2n}})) > 1$ if $a \in A_{p, n}$ (see \cite[Lemmas 3.1 and 3.2]{Sa93}).
The cardinality of $A_{p, n}$ is $2(p - 1)p^{n-2}$ (see the proof of Theorem 3.3 in \cite{Sa93}).
He counted the number of imaginary quadratic fields $\mathbb{Q}(\sqrt{a^2 - 4p^{2n}})$ with $a \in A_{p, n}$ and
showed that there exist at least $2(p - 1)p^{n-2} - 3$ imaginary quadratic fields $K$ such that $p$ splits in $K$, $\lambda_p(K) > 1$,
 and the fundamental discriminant of $K$ is greater than $-4p^{2n}$ (see \cite[Theorem 3.3]{Sa93}).
Let $n \rightarrow \infty$ in the above result.
Then, we find that there exist infinitely many imaginary quadratic fields $K$ such that $p$ splits in $K$
 and the Iwasawa $\lambda_p$-invariant of $K$ is greater than $1$ (see \cite[Corollary 3.4]{Sa93}).\\
\hspace{16pt}We refine his proof.
By studying the divisibility of the class number of the imaginary quadratic fields $\mathbb{Q}(\sqrt{1 - 4p^n})$, 
we obtain a simpler proof of his result in the following way.
\begin{thm} \label{thm14}
Let $p$ be an odd prime number and let $n$ be an integer greater than $1$ such that $\gcd(p, n) = 1$.
Then, 
$$\lambda_p(\mathbb{Q}(\sqrt{1 - 4p^n})) > 1.$$
\end{thm}
\begin{thm} \label{thm14.1}
Let $p$ be an odd prime number.
If $n_1$ and $n_2$ are integers greater than $8$
such that $\mathbb{Q}(\sqrt{1 - 4p^{n_1}}) = \mathbb{Q}(\sqrt{1 - 4p^{n_2}})$, then $n_1 = n_2$.
\end{thm}
Theorem \ref{thm14.1} follows from our proof that the order of the ideal class containing the prime ideal over $p$ is $n$.
By Theorems \ref{thm14} and \ref{thm14.1}, for a fixed $p$, the set
$$\{\mathbb{Q}(\sqrt{1 - 4p^{n'}}) \mid n' \in \mathbb{Z} \ such \ that \ n' > 8 \ and \ \gcd(p, n') = 1\}$$
contains infinitely many imaginary quadratic fields whose Iwasawa $\lambda_p$-invari-\\
ant is greater than $1$.     

In the above results, we construct infinite families of imaginary quadratic fields with $\lambda_p > 1$ explicitly.
On the other hand, in Theorems \ref{thm3} and \ref{thm8}, we do not construct infinite families of imaginary quadratic fields with $\lambda_p = 1$ explicitly
(see the proof of Theorem \ref{thm8} in Section \ref{sec3} of this paper). 
Finally, we consider the following question.
\begin{prob} \label{q11}
Let $p$ be an odd prime number.
Construct explicitly an infinite family of imaginary quadratic fields whose Iwasawa $\lambda_p$-invariant is equal to $1$.
\end{prob}
We construct such imaginary quadratic fields in the following theorem.
However, we do not know whether they are infinite or not.
We denote by $h(k)$ the class number of an algebraic number field $k$.
\begin{thm} \label{thm12}
Let $p$ be an odd prime number, $q_1$ a prime factor of $p - 2$ such that $q_1^{p-1} \not \equiv 1 \bmod p^2$, and
$n$ an integer greater than $1$ such that $\gcd(p, n) = 1$.\\
$(1)$ Assume $p \equiv 3 \bmod 4$.\\
$({\rm i})$ Suppose $2^{p-1} \not \equiv 1 \bmod p^2$. 
If $p \nmid h(\mathbb{Q}(\sqrt{1 - p^n}))$, then 
$$\lambda_p(\mathbb{Q}(\sqrt{1 - p^n})) = 1.$$
$({\rm ii})$ Suppose $2^{p-1} \equiv 1 \bmod p^2$. 
If $p \nmid h(\mathbb{Q}(\sqrt{q_1^2 - p^n}))$, then 
$$\lambda_p(\mathbb{Q}(\sqrt{q_1^2 - p^n})) = 1.$$ 
$(2)$ Assume $p \equiv 1 \bmod 4$.\\
$({\rm i})$ Suppose $2^{p-1} \not \equiv 1 \bmod p^2$. 
If $p \nmid h(\mathbb{Q}(\sqrt{4 - p^n}))$, then 
$$\lambda_p(\mathbb{Q}(\sqrt{4 - p^n})) = 1.$$
$({\rm ii})$ Suppose $2^{p-1} \equiv 1 \bmod p^2$. 
If $p \nmid h(\mathbb{Q}(\sqrt{4q_1^2 - p^n}))$, then 
$$\lambda_p(\mathbb{Q}(\sqrt{4q_1^2 - p^n})) = 1.$$
\end{thm}
\begin{rem} \label{rem13}
A prime number $p$ such that $2^{p-1} \equiv 1 \bmod p^2$ is called a Wieferich prime.
As examples of Wieferich primes, 1093 and 3511 are known.
\end{rem}
\begin{rem} \label{rem14}
We give a remark on Theorem \ref{thm12} (1) (ii) and (2) (ii).
For a given odd prime number $p$ with $2^{p-1} \equiv 1 \bmod p^2$, we can prove the existence of $q_1$ as follows.
Assume that $q_1^{p-1} \equiv 1 \bmod p^2$ for any prime factor $q_1$ of $p - 2$.
Then, $(p - 2)^{p-1} \equiv 1 \bmod p^2$.
Expanding the left side of this equation, we find
\begin{equation}
 \begin{split}
(p - 2)^{p-1} &= \sum_{j=0}^{p-1}\binom {p-1}{j}p^j(-2)^{p-1-j} \notag \\
                      &= (-2)^{p-1} + (p - 1)p(-2)^{p-2} + \sum_{j=2}^{p-1}\binom {p-1}{j}p^j(-2)^{p-1-j} \notag \\
                      &\equiv 2^{p-1} - p(-2)^{p-2} \bmod p^2, \notag
 \end{split}
\end{equation}                       
where $\binom {p}{j}$ denotes the binomial coefficient.
Using the assumption $2^{p-1} \equiv 1 \bmod p^2$, we have
$$2^{p-1} - p(-2)^{p-2} \equiv 1 - p(-2)^{p-2}\bmod p^2.$$
Then,
$$1 \equiv 1 - p(-2)^{p-2}\bmod p^2,$$
that is,
$$p \mid 2.$$
This is a contradiction.
Then, there exists at least one prime factor $q_1$ of $p - 2$ such that $q_1^{p-1} \not \equiv 1 \bmod p^2$.
\end{rem}
For the imaginary quadratic fields treated in Theorem \ref{thm12}, we obtain the following theorem.
\begin{thm} \label{thm12.1}
Let $p$ be an odd prime number and let $q_1$ be a prime factor of $p - 2$ such that $q_1^{p-1} \not \equiv 1 \bmod p^2$.\\
$(1)$ Assume $p \equiv 3 \bmod 4$.\\
\hspace{10pt}$({\rm i})$ Suppose $p \neq 3$ and $2^{p-1} \not \equiv 1 \bmod p^2$. 
If $n_1$ and $n_2$ are positive odd integers such that $\mathbb{Q}(\sqrt{1 - p^{n_1}}) = \mathbb{Q}(\sqrt{1 - p^{n_2}})$, 
then $n_1 = n_2$.\\
\hspace{10pt}$({\rm ii})$ Suppose $p =  3$. 
If $n_1$ and $n_2$ are positive odd integers with $n_1$, $n_2 \neq 5$ such that $\mathbb{Q}(\sqrt{1 - p^{n_1}}) = \mathbb{Q}(\sqrt{1 - p^{n_2}})$, 
then $n_1 = n_2$.\\
\hspace{10pt}$({\rm iii})$ Suppose $2^{p-1} \equiv 1 \bmod p^2$.
If $n_1$ and $n_2$ are positive odd composite numbers such that $\mathbb{Q}(\sqrt{q_1^2 - p^{n_1}}) = \mathbb{Q}(\sqrt{q_1^2 - p^{n_2}})$, 
then $n_1 = n_2$.\\
$(2)$ Assume $p \equiv 1 \bmod 4$.\\
\hspace{10pt}$({\rm i})$ Suppose $2^{p-1} \not \equiv 1 \bmod p^2$. 
If $n_1$ and $n_2$ are positive integers such that $\mathbb{Q}(\sqrt{4 - p^{n_1}})$, $\mathbb{Q}(\sqrt{4 - p^{n_2}}) \neq \mathbb{Q}(\sqrt{-1})$
and $\mathbb{Q}(\sqrt{4 - p^{n_1}}) = \mathbb{Q}(\sqrt{4 - p^{n_2}})$, then $n_1 = n_2$.\\
\hspace{10pt}$({\rm ii})$ Suppose $2^{p-1} \equiv 1 \bmod p^2$. 
If $n_1$ and $n_2$ are positive composite numbers such that $\mathbb{Q}(\sqrt{4q_1^2 - p^{n_1}})$, $\mathbb{Q}(\sqrt{4q_1^2 - p^{n_2}}) \neq \mathbb{Q}(\sqrt{-1})$
and $\mathbb{Q}(\sqrt{4q_1^2 - p^{n_1}}) = \mathbb{Q}(\sqrt{4q_1^2 - p^{n_2}})$, then $n_1 = n_2$.
\end{thm}
This theorem follows from our proof that the order of the ideal class containing the prime ideal over $p$ is $n$.
    
To solve Question \ref{q11} by using Theorems \ref{thm12} and \ref{thm12.1},
we need to show that there exist infinitely many imaginary quadratic fields treated in these theorems
whose class number is indivisible by $p$.
As a further question, we give this.
We give some examples of Theorem \ref{thm12}.\\
~\\
Case: $p = 3$\\
~\\
\begin{tabular}{|c|c|c|}
\hline
$n$ & $\mathbb{Q}(\sqrt{1 - 3^n})$ & ideal \ class \ group \ of \ $\mathbb{Q}(\sqrt{1 - 3^n})$ \\
\hline
&&\\[-14pt]\hline  
$2$ & $\mathbb{Q}(\sqrt{-2})$ & trivial \\ 
\hline   
$4$ & $\mathbb{Q}(\sqrt{-5})$ & $\mathbb{Z}/2\mathbb{Z}$ \\ 
\hline      
$5$ & $\mathbb{Q}(\sqrt{-2})$ & trivial \\ 
\hline
$7$ & $\mathbb{Q}(\sqrt{-2186})$ & $\mathbb{Z}/42\mathbb{Z}$ \\ 
\hline
$8$ & $\mathbb{Q}(\sqrt{-410})$ & $\mathbb{Z}/2\mathbb{Z} \times \mathbb{Z}/8\mathbb{Z}$ \\ 
\hline
$10$ & $\mathbb{Q}(\sqrt{-122})$ & $\mathbb{Z}/10\mathbb{Z}$ \\ 
\hline
$11$ & $\mathbb{Q}(\sqrt{-177146})$ & $\mathbb{Z}/2\mathbb{Z} \times \mathbb{Z}/198\mathbb{Z}$ \\ 
\hline
$13$ & $\mathbb{Q}(\sqrt{-1594322})$ & $\mathbb{Z}/780\mathbb{Z}$\\ 
\hline
$14$ & $\mathbb{Q}(\sqrt{-1195742})$ & $\mathbb{Z}/2\mathbb{Z} \times \mathbb{Z}/322\mathbb{Z}$\\ 
\hline
$16$ & $\mathbb{Q}(\sqrt{-672605})$ & $\mathbb{Z}/2\mathbb{Z} \times \mathbb{Z}/2\mathbb{Z} \times 
\mathbb{Z}/2\mathbb{Z} \times \mathbb{Z}/112\mathbb{Z}$\\ 
\hline
$17$ & $\mathbb{Q}(\sqrt{-129140162})$ & $\mathbb{Z}/2\mathbb{Z} \times \mathbb{Z}/5304\mathbb{Z}$ \\ 
\hline
$19$ & $\mathbb{Q}(\sqrt{-1162261466})$ & $\mathbb{Z}/2\mathbb{Z} \times \mathbb{Z}/16074\mathbb{Z}$ \\ 
\hline
$20$ & $\mathbb{Q}(\sqrt{-72041})$ & $\mathbb{Z}/2\mathbb{Z} \times \mathbb{Z}/140\mathbb{Z}$ \\ 
\hline
\end{tabular}\\
~\\
We take $n$ as integers greater than $1$ such that $\gcd(3, n) = 1$.
For $n = 2$, $4$, $5$, $8$, $10$, $14$, $16$, $20$, the class number of the imaginary quadratic fields $\mathbb{Q}(\sqrt{1 - 3^n})$
 is indivisible by $3$.
For $n = 7$, $11$, $13$, $17$, $19$, the class number of the imaginary quadratic fields $\mathbb{Q}(\sqrt{1 - 3^n})$
 is divisible by $3$.\\
~\\
Case: $p = 5$\\
~\\
\begin{tabular}{|c|c|c|}
\hline
$n$ & $\mathbb{Q}(\sqrt{4 - 5^n})$ & ideal \ class \ group \ of \ $\mathbb{Q}(\sqrt{4 - 5^n})$ \\
\hline
&&\\[-14pt]\hline 
$2$ & $\mathbb{Q}(\sqrt{-21})$ & $\mathbb{Z}/2\mathbb{Z} \times \mathbb{Z}/2\mathbb{Z}$ \\ 
\hline     
\end{tabular}\\
\begin{tabular}{|c|c|c|}
\hline
$n$ & $\mathbb{Q}(\sqrt{4 - 5^n})$ & ideal \ class \ group \ of \ $\mathbb{Q}(\sqrt{4 - 5^n})$ \\
\hline
&&\\[-14pt]\hline 
$3$ & $\mathbb{Q}(\sqrt{-1})$ & trivial \\ 
\hline
$4$ & $\mathbb{Q}(\sqrt{-69})$ & $\mathbb{Z}/2\mathbb{Z} \times \mathbb{Z}/4\mathbb{Z}$ \\ 
\hline
$6$ & $\mathbb{Q}(\sqrt{-15621})$ & $\mathbb{Z}/2\mathbb{Z} \times \mathbb{Z}/2\mathbb{Z} \times \mathbb{Z}/18\mathbb{Z}$ \\ 
\hline
$7$ & $\mathbb{Q}(\sqrt{-78121})$ & $\mathbb{Z}/168\mathbb{Z}$ \\ 
\hline
$8$ & $\mathbb{Q}(\sqrt{-390621})$ & $\mathbb{Z}/2\mathbb{Z} \times \mathbb{Z}/2\mathbb{Z} \times 
\mathbb{Z}/2\mathbb{Z} \times \mathbb{Z}/8\mathbb{Z} \times \mathbb{Z}/8\mathbb{Z}$ \\ 
\hline
$9$ & $\mathbb{Q}(\sqrt{-1953121})$ & $\mathbb{Z}/2\mathbb{Z} \times \mathbb{Z}/360\mathbb{Z}$ \\ 
\hline
$11$ & $\mathbb{Q}(\sqrt{-48828121})$ & $\mathbb{Z}/2\mathbb{Z} \times \mathbb{Z}/2\mathbb{Z} \times \mathbb{Z}/1188\mathbb{Z}$ \\ 
\hline
$12$ & $\mathbb{Q}(\sqrt{-244140621})$ & $\mathbb{Z}/2\mathbb{Z} \times \mathbb{Z}/2\mathbb{Z} \times \mathbb{Z}/2\mathbb{Z} \times \mathbb{Z}/1620\mathbb{Z}$ \\ 
\hline
$13$ & $\mathbb{Q}(\sqrt{-1220703121})$ & $\mathbb{Z}/2\mathbb{Z} \times \mathbb{Z}/10946\mathbb{Z}$ \\ 
\hline
\end{tabular}\\
~\\
We take $n$ as integers greater than $1$ such that $\gcd(5, n) = 1$.
For $n = 2$, $3$, $4$, $6$, $7$, $8$, $11$, $13$, the class number of the imaginary quadratic fields $\mathbb{Q}(\sqrt{4 - 5^n})$
 is indivisible by $5$.
For $n = 9$, $12$, the class number of the imaginary quadratic fields $\mathbb{Q}(\sqrt{4 - 5^n})$
 is divisible by $5$.\\
~\\
Case: $p = 7$\\
~\\
\begin{tabular}{|c|c|c|}
\hline
$n$ & $\mathbb{Q}(\sqrt{1 - 7^n})$ & ideal \ class \ group \ of \ $\mathbb{Q}(\sqrt{1 - 7^n})$ \\
\hline
&&\\[-14pt]\hline
$2$ & $\mathbb{Q}(\sqrt{-3})$ & trivial \\
\hline
$3$ & $\mathbb{Q}(\sqrt{-38})$ & $\mathbb{Z}/6\mathbb{Z}$ \\
\hline
$4$ & $\mathbb{Q}(\sqrt{-6})$ & $\mathbb{Z}/2\mathbb{Z}$ \\
\hline
$5$ & $\mathbb{Q}(\sqrt{-16806})$ & $\mathbb{Z}/2\mathbb{Z} \times \mathbb{Z}/50\mathbb{Z}$ \\ 
\hline
$6$ & $\mathbb{Q}(\sqrt{-817})$ & $\mathbb{Z}/2\mathbb{Z} \times \mathbb{Z}/6\mathbb{Z}$ \\ 
\hline
$8$ & $\mathbb{Q}(\sqrt{-3603})$ & $\mathbb{Z}/16\mathbb{Z}$ \\ 
\hline
$9$ & $\mathbb{Q}(\sqrt{-4483734})$ & $\mathbb{Z}/2\mathbb{Z} \times \mathbb{Z}/2\mathbb{Z} \times \mathbb{Z}/2\mathbb{Z} \times \mathbb{Z}/234\mathbb{Z}$ \\ 
\hline
$10$ & $\mathbb{Q}(\sqrt{-17654703})$ & $\mathbb{Z}/2\mathbb{Z} \times \mathbb{Z}/2\mathbb{Z} \times \mathbb{Z}/780\mathbb{Z}$ \\ 
\hline
$11$ & $\mathbb{Q}(\sqrt{-1977326742})$ & $\mathbb{Z}/2\mathbb{Z} \times \mathbb{Z}/2\mathbb{Z} \times \mathbb{Z}/9438\mathbb{Z}$ \\ 
\hline
$12$ & $\mathbb{Q}(\sqrt{-3844802})$ & $\mathbb{Z}/2\mathbb{Z} \times \mathbb{Z}/2\mathbb{Z} \times \mathbb{Z}/4\mathbb{Z} \times \mathbb{Z}/96\mathbb{Z}$ \\ 
\hline
\end{tabular}\\
~\\
We take $n$ as integers greater than $1$ such that $\gcd(7, n) = 1$.
For the above $n$'s, the class number of the imaginary quadratic fields $\mathbb{Q}(\sqrt{1 - 7^n})$
 is indivisible by $7$.     

This paper is organized as follows.
In Section \ref{sec2}, we give a proof of Theorem \ref{thm4} by using a result of Sands~\cite{Sa93}.
In Section \ref{sec3}, we prove Theorem \ref{thm8} by using some properties of the Fourier coefficients of Cohen's Eisenstein series of half integral weight \cite{Co75}.
In Section \ref{sec4}, we show Theorems \ref{thm14} and \ref{thm12} by using the same method of the proof of Theorem \ref{thm4}
and show Theorems \ref{thm14.1} and \ref{thm12.1} by using a method in the study of divisibility of the class number of imaginary quadratic fields.
\section{Proof of Theorem \ref{thm4}} \label{sec2}
In this section, we show Theorem \ref{thm4}.
The method of the proof is based on the one in \cite{By05}.
To check whether the Iwasawa $\lambda$-invariants of the cyclotomic $\mathbb{Z}_p$-extensions of imaginary quadratic fields are equal to $1$, we use the following theorem.
\begin{thm}[Sands, {[26, Proposition 2.1]}] \label{thm1.1}
Assume that $p$ is an odd prime number and that $p$ splits in the imaginary quadratic field $K$. 
Thus $(p) = \wp \overline{\wp}$, the product of prime ideals of $K$. 
Suppose that $m_1$ is a positive integer not divisible by $p$ such that $\wp^{m_1} = (\xi)$, a principal ideal of $K$. 
Then $\lambda_p(K) > 1$ if and only if either $\xi^{p-1} \equiv 1 \bmod \overline{\wp}^2$ or $p$ divides the class number of $K$.
\end{thm}
\begin{rem}
We see from Theorem \ref{thm1.1} that $\lambda_p(K) = 1$ 
if and only if $\xi^{p-1} \not \equiv 1 \bmod \overline{\wp}^2$ and $p$ does not divide the class number of $K$.
When $p$ does not divide the class number of $K$,
$\lambda_p(K) = 1$ if and only if $\xi^{p-1} \not \equiv 1 \bmod \overline{\wp}^2$.
Gold proved this special case in \cite{Go74} and Sands improved his result as seen in the above.
To use their necessary and sufficient condition for $\lambda_p(K) = 1$, 
we need to check that $p$ does not divide the class number of $K$.
However, Byeon~\cite{By05} does not verify that 
his constructed imaginary quadratic fields satisfy this class number condition.
\end{rem}
To use Theorem \ref{thm1.1}, we need the following lemma.
\begin{lem} \label{lem1.2}
Let $p$ be an odd prime number such that $p \ge e^7$ and $x_1$ a positive integer such that $x_1^2 < p$.
Then, $p \nmid h(\mathbb{Q}(\sqrt{x_1^2 - p}))$. 
\end{lem}
\begin{proof}
Let $K$ be an imaginary quadratic field.
The class number formula of imaginary quadratic fields is as follows:
$$h(K) = \frac{\omega_K\sqrt{| D_K |}}{2\pi}L(1, \chi_{D_K}),$$
where $\omega_K$ denotes the number of roots of unity in $K$,
$D_K$ denotes the fundamental discriminant of $K$, and $L(s, \chi_{D_K})$ denotes the Dirichlet $L$-function.
We substitute $K = \mathbb{Q}(\sqrt{x_1^2 - p})$ in this formula.
Since $p \nmid h(\mathbb{Q}(\sqrt{-1})) = 1$ and $p \nmid h(\mathbb{Q}(\sqrt{-3})) = 1$ hold, 
we may assume $\mathbb{Q}(\sqrt{x_1^2 - p}) \neq \mathbb{Q}(\sqrt{-1})$, $\mathbb{Q}(\sqrt{-3})$,
that is, we may assume $\omega_{\mathbb{Q}(\sqrt{x_1^2 - p})} = 2$.  
Substituting $\omega_{\mathbb{Q}(\sqrt{x_1^2 - p})} = 2$ in the above class number formula, we have
$$h(\mathbb{Q}(\sqrt{x_1^2 - p})) = \frac{\sqrt{| D_{\mathbb{Q}(\sqrt{x_1^2 - p})} |}}{\pi}L(1, \chi_{D_{\mathbb{Q}(\sqrt{x_1^2 - p})}}).$$
From
$$| D_{\mathbb{Q}(\sqrt{x_1^2 - p})} | \le 4 | x_1^2 - p | = 4(p - x_1^2) < 4p,$$
we see
$$h(\mathbb{Q}(\sqrt{x_1^2 - p})) < \frac{2\sqrt{p}}{\pi}L(1, \chi_{D_{\mathbb{Q}(\sqrt{x_1^2 - p})}}).$$
The value of $L(1, \chi_{D_K})$ satisfies the inequality
$$L(1, \chi_{D_K}) \le \frac{1}{2}\log | D_K | + \log \log | D_K | + 2.8$$
(see \cite[Proposition 10.3.16]{Co}).
Using this inequality, we obtain
\begin{equation} \label{eq1.1}
 \begin{split}
\frac{2\sqrt{p}}{\pi}L(1, &\chi_{D_{\mathbb{Q}(\sqrt{x_1^2 - p})}}) \\
                           &\le \frac{2\sqrt{p}}{\pi}\biggl\{\frac{1}{2}\log | D_{\mathbb{Q}(\sqrt{x_1^2 - p})} | + \log \log | D_{\mathbb{Q}(\sqrt{x_1^2 - p})} | + 2.8\biggr\} \\  
                           &< \frac{2\sqrt{p}}{\pi}\biggl\{\frac{1}{2}\log 4p + \log (\log 4p) + 2.8\biggr\} \\
                           &= \frac{2\sqrt{p}}{\pi}\biggl\{\frac{1}{2}\cdot 2\log 2 + \frac{1}{2}\log p + \log (2\log 2 + \log p) + 2.8\biggr\} \\ 
                           &= \frac{2\sqrt{p}}{\pi}\biggl\{\frac{1}{2}\log p + (\log 2 +  \log (2\log 2 + \log p) + 2.8)\biggr\} \\  
                           &< \frac{2\sqrt{p}}{\pi}\biggl\{\frac{1}{2}\log p + ( 4 +  \log (2\log 2 + \log p))\biggr\}. \\                                                                                                         
 \end{split}
\end{equation}                       
When $p \ge e^7$, the inequality
$$4 +  \log (2\log 2 + \log p) < \log p$$
holds.
In fact, we can show this as follows.
Let
$$f_1(X) := \log X - 4 - \log (2\log 2 + \log X).$$
From 
$$f_1^{'}(X) = \frac{1}{X} - \frac{1 / X}{2\log 2 + \log X} = \frac{1}{X}\biggl(\frac{2\log 2 + \log X - 1}{2\log 2 + \log X}\biggr),$$
we see that $f_1^{'}(X) > 0$ when $X \ge e$.
Then, $f_1(X)$ increases monotonously when $X \ge e$. 
When $X = e^7$, we have
\begin{equation}
 \begin{split}
f_1(e^7) &= \log e^7 - 4 - \log (2\log 2 + \log e^7) \notag \\
           &= 7 - 4 - \log (2\log 2 + 7) \notag \\
           &= 3 - \log (2\log 2 + 7) \notag \\
           &> 3 - \log (2 + 7) = 3 - \log 9 = \log \frac{e^3}{9} > 0. \notag \\                                                                                                   
 \end{split}
\end{equation}                       
Then, $f_1(X) > 0$ when $X \ge e^7$, that is, 
$$\log X > 4 + \log (2\log 2 + \log X)$$
 when $X \ge e^7$.
Substituting this in equation (\ref{eq1.1}), we see
\begin{equation} \label{eq1.2}
 \begin{split}
\frac{2\sqrt{p}}{\pi}L(1, \chi_{D_{\mathbb{Q}(\sqrt{x_1^2 - p})}}) &< \frac{2\sqrt{p}}{\pi}\biggl(\frac{1}{2}\log p + \log p\biggr) \\
                           &= \frac{2\sqrt{p}}{\pi} \frac{3}{2} \log p = \frac{3\sqrt{p}}{\pi} \log p \\                                                                                                           
 \end{split}
\end{equation}                       
when $p \ge e^7$.    
When $p \ge e^2$, the inequality $\log p < \sqrt{p}$ holds.
In fact, we can check this as follows.
Let $f_2(X) := \sqrt{X} - \log X$.
Since
$$f_2^{'}(X) = \frac{1}{2}X^{-\frac{1}{2}} - \frac{1}{X} = \frac{1}{2\sqrt{X}} - \frac{1}{X} = \frac{\sqrt{X} - 2}{2X}$$
holds, we see $f_2^{'}(X) > 0$ when $X > 4$.
Then, $f_2(X)$ increases monotonously when $X > 4$.  
From 
$$f_2(e^2) = \sqrt{e^2} - \log e^2 = e - 2 > 0,$$
$f_2(X) > 0$ holds when $X \ge e^2$.
Hence, $\log p < \sqrt{p}$ when $p \ge e^2$.                           
Substituting this in equation (\ref{eq1.2}), we have
\begin{equation}
 \begin{split}
\frac{2\sqrt{p}}{\pi}L(1, \chi_{D_{\mathbb{Q}(\sqrt{x_1^2 - p})}}) &< \frac{3\sqrt{p}}{\pi} \log p \notag \\    
                                                                                                           &< \frac{3\sqrt{p}}{\pi} \sqrt{p} = \frac{3p}{\pi} < p \notag \\   
 \end{split}
\end{equation}                       
when $p \ge e^7$.                           
This implies that $h(\mathbb{Q}(\sqrt{x_1^2 - p})) < p$ when $p \ge e^7$, that is, $p \nmid h(\mathbb{Q}(\sqrt{x_1^2 - p}))$ when $p \ge e^7$.
The proof of Lemma \ref{lem1.2} is completed.
\end{proof}
Using Kash program, we can check that the class numbers of the imaginary quadratic fields $\mathbb{Q}(\sqrt{1 - p})$ and $\mathbb{Q}(\sqrt{4 - p})$
are not divisible by $p$ when $3 < p < e^7$.
From this and Lemma \ref{lem1.2}, we obtain the following lemma.
\begin{lem} \label{lem1.3}
Let $p$ be an odd prime number greater than $3$.
Then, the class numbers of the imaginary quadratic fields $\mathbb{Q}(\sqrt{1 - p})$ and $\mathbb{Q}(\sqrt{4 - p})$
are not divisible by $p$. 
\end{lem}
We show Theorem \ref{thm4} by using Lemma \ref{lem1.3}.\\
~\\
{\bf Proof of Theorem \ref{thm4}.} \ \ First, we treat $\mathbb{Q}(\sqrt{1 - p})$.
Suppose $\lambda_p(\mathbb{Q}(\sqrt{1 - p})) > 1$.
We can write
$$\wp_1\overline{\wp_1} = (p) = (1 + \sqrt{1 - p})(1 - \sqrt{1 - p})$$
in $\mathbb{Q}(\sqrt{1 - p})$,
where $\wp_1$ denotes the prime ideal of $\mathbb{Q}(\sqrt{1 - p})$ over $p$ and $\overline{\wp_1}$ denotes the complex conjugate of $\wp_1$.
Since $p \nmid (1 + \sqrt{1 - p})$ and $p \nmid (1 - \sqrt{1 - p})$ hold, we may assume $\wp_1 = (1 + \sqrt{1 - p})$ and $\overline{\wp_1} = (1 - \sqrt{1 - p})$.
Then, $\wp_1^2$ is a principal ideal.
Hence it follows from Theorem \ref{thm1.1} and Lemma \ref{lem1.3} that 
$\lambda_p(\mathbb{Q}(\sqrt{1 - p})) > 1$ if and only if
$$((1 + \sqrt{1 - p})^2)^{p-1} \equiv 1 \bmod \overline{\wp_1}^2.$$
This is equivalent to
$$(1 + \sqrt{1 - p})^{2(p-1)}(1 + \sqrt{1 - p})^2 - (1 + \sqrt{1 - p})^2 \equiv 0 \bmod \wp_1^2\overline{\wp_1}^2,$$
that is, 
\begin{equation} \label{eq1.3}
(1 + \sqrt{1 - p})^{2p} - (1 + \sqrt{1 - p})^2 \equiv 0 \bmod p^2.
\end{equation}
Expanding the left side of equation (\ref{eq1.3}), we obtain
\begin{equation} \label{eq1.4}
 \begin{split}
&(1 + \sqrt{1 - p})^{2p} - (1 + \sqrt{1 - p})^2 \\
= &\biggl\{\sum_{j = 0}^p \binom{2p}{2j}(1 - p)^j\biggr\} + \sqrt{1 - p} \biggl\{\sum_{j = 0}^{p-1} \binom{2p}{2j+1}(1 - p)^j\biggr\} \\
&- (2 - p + 2\sqrt{1 - p}). \\
 \end{split}
\end{equation}                       
Note that
\begin{equation}
 \begin{split}
(1 - p)^j &= \sum_{i = 0}^j \binom{j}{i}(-p)^i = 1 + \binom{j}{1}(-p) + \sum_{i = 2}^j \binom{j}{i}(-p)^i \notag \\
               &\equiv 1 - jp \bmod p^2. \notag \\
 \end{split}
\end{equation}                       
Substituting this in equation (\ref{eq1.4}), we see
\begin{equation}
 \begin{split}
&(1 + \sqrt{1 - p})^{2p} - (1 + \sqrt{1 - p})^2 \notag \\
\equiv &\biggl\{\sum_{j = 0}^p \binom{2p}{2j}(1 - jp)\biggr\} + \sqrt{1 - p} \biggl\{\sum_{j = 0}^{p-1} \binom{2p}{2j+1}(1 - jp)\biggr\}  \notag \\
&- (2 - p + 2\sqrt{1 - p})  \notag \\
\equiv &\sum_{j = 0}^p \binom{2p}{2j} - p \sum_{j = 0}^p \binom{2p}{2j}j + \sqrt{1 - p} \biggl\{\sum_{j = 0}^{p-1} \binom{2p}{2j+1} - p\sum_{j = 0}^{p-1} \binom{2p}{2j+1}j  
\biggr\} \notag \\
& - (2 - p + 2\sqrt{1 - p}) \bmod p^2. 
 \end{split}
\end{equation}                       
Since
$$\sum_{j = 0}^p \binom{2p}{2j} = \sum_{j = 0}^{p-1} \binom{2p}{2j+1} = 2^{2p-1},\hspace{10pt}\sum_{j = 0}^p \binom{2p}{2j} j = 2^{2p-2}p,$$
and
$$\sum_{j = 0}^{p-1} \binom{2p}{2j+1} j = 2^{2p-2}(p - 1)$$
hold, we have
\begin{equation}
 \begin{split}
&(1 + \sqrt{1 - p})^{2p} - (1 + \sqrt{1 - p})^2 \notag \\
\equiv &\hspace{5pt}2^{2p-1} - 2^{2p-2}p^2 + \sqrt{1 - p} \{2^{2p-1} - 2^{2p-2}p(p - 1)\} - (2 - p + 2\sqrt{1 - p}) \notag \\
\equiv &\hspace{5pt}2^{2p-1} + \sqrt{1 - p} (2^{2p-1} + 2^{2p-2}p) - (2 - p + 2\sqrt{1 - p}) \notag \\
\equiv &\hspace{5pt}(2^{2p-1} - 2 + p) + \sqrt{1 - p} (2^{2p-1} + 2^{2p-2}p - 2) \bmod p^2. 
 \end{split}
\end{equation}                       
From this and equation (\ref{eq1.3}), we obtain
$$2^{2p-1} - 2 + p \equiv 0 \bmod p^2$$
and
$$2^{2p-1} + 2^{2p-2}p - 2 \equiv 0 \bmod p^2.$$
This is equivalent to $2^{2p-1} - 2 + p \equiv 0 \bmod p^2$.
In fact, when $2^{2p-1} \equiv 2 - p \bmod p^2$, we see
\begin{equation}
 \begin{split}
2^{2p-1} + 2^{2p-2}p - 2 &\equiv 2 - p + 2^{2p-2}p - 2 \notag \\ 
&\equiv - p + 2^{2p-2}p \equiv p(2^{2p-2} - 1) \bmod p^2. \notag \\
 \end{split}
\end{equation}                       
Since $2^{2p-2} - 1 \equiv 0 \bmod p$ holds,
we have 
$$2^{2p-1} + 2^{2p-2}p - 2 \equiv 0 \bmod p^2.$$
Then, $\lambda_p(\mathbb{Q}(\sqrt{1 - p})) > 1$ if and only if
$$2^{2p-1} - 2 + p \equiv 0 \bmod p^2.$$
Secondly, we treat $\mathbb{Q}(\sqrt{4 - p})$.
We can write
$$\wp_2\overline{\wp_2} = (p) = (2 + \sqrt{4 - p})(2 - \sqrt{4 - p})$$
in $\mathbb{Q}(\sqrt{4 - p})$,
where $\wp_2$ denotes the prime ideal of $\mathbb{Q}(\sqrt{4 - p})$ over $p$ and $\overline{\wp_2}$ denotes the complex conjugate of $\wp_2$.
Since $p \nmid (2 + \sqrt{4 - p})$ and $p \nmid (2 - \sqrt{4 - p})$ hold, we may assume $\wp_2 = (2 + \sqrt{4 - p})$ and $\overline{\wp_2} = (2 - \sqrt{4 - p})$.
Then, $\wp_2^2$ is a principal ideal.
Hence it follows from Theorem \ref{thm1.1} and Lemma \ref{lem1.3} that 
$\lambda_p(\mathbb{Q}(\sqrt{4 - p})) = 1$ if and only if
\begin{equation} \label{eq1.5}
((2 + \sqrt{4 - p})^2)^{p-1} \not \equiv 1 \bmod \overline{\wp_2}^2.
\end{equation}                       
This is equivalent to
$$2^{4p-1} - 2^3 + p \not \equiv 0 \bmod p^2$$
or
$$2^{4p-2} + 2^{4p-5}p - 2^2 \not \equiv 0 \bmod p^2.$$
We can check this as follows.
Equation (\ref{eq1.5}) is equivalent to
$$(2 + \sqrt{4 - p})^{2(p-1)}(2 + \sqrt{4 - p})^2 - (2 + \sqrt{4 - p})^2 \not \equiv 0 \bmod \wp_2^2\overline{\wp_2}^2,$$
that is, 
\begin{equation} \label{eq1.6}
(2 + \sqrt{4 - p})^{2p} - (2 + \sqrt{4 - p})^2 \not \equiv 0 \bmod p^2.
\end{equation}
Expanding the left side of equation (\ref{eq1.6}), we obtain
\begin{equation}
 \begin{split}
&(2 + \sqrt{4 - p})^{2p} - (2 + \sqrt{4 - p})^2 \notag \\ 
= &\sum_{j=0}^p \binom{2p}{2j} (4 - p)^j 2^{2p-2j} + \sqrt{4 - p} \biggl\{\sum_{j=0}^{p-1} \binom{2p}{2j+1}(4 - p)^j 2^{2p-(2j+1)}\biggr\} \notag \\
   &- (4 + 4\sqrt{4 - p} + 4 - p) \notag \\
\equiv &\sum_{j=0}^p \binom{2p}{2j} (4^j - 4^{j-1}jp)2^{2p-2j} \notag \\
   &+ \sqrt{4 - p} \biggl\{\sum_{j=0}^{p-1} \binom{2p}{2j+1} (4^j - 4^{j-1}jp)2^{2p-2j-1}\biggr\} - (8 - p + 4\sqrt{4 - p}) \notag \\   
 \end{split}
\end{equation}                       
\begin{equation}
 \begin{split}   
\equiv &\sum_{j=0}^p \binom{2p}{2j} (2^{2p} - 2^{2p-2}jp) + \sqrt{4 - p} \biggl\{\sum_{j=0}^{p-1} \binom{2p}{2j+1} (2^{2p-1} - 2^{2p-3}jp)\biggr\} \notag \\
   &- (8 - p + 4\sqrt{4 - p}) \notag \\    
\equiv &\hspace{5pt}2^{2p} \sum_{j=0}^p \binom{2p}{2j} - 2^{2p-2}p \sum_{j=0}^p j \binom{2p}{2j} \notag \\
   &+ \sqrt{4 - p} \biggl\{2^{2p-1} \sum_{j=0}^{p-1} \binom{2p}{2j+1} - 2^{2p-3}p \sum_{j=0}^{p-1} j \binom{2p}{2j+1}\biggr\} \notag \\
   &- (8 - p + 4\sqrt{4 - p}) \notag \\
\equiv &\hspace{5pt}2^{2p}2^{2p-1} - 2^{2p-2}p \cdot 2^{2p-2}p \notag \\
   &+ \sqrt{4 - p} \{2^{2p-1}2^{2p-1} - 2^{2p-3}p (p - 1) 2^{2p-2}\} - 8 + p -4\sqrt{4 - p} \notag \\
\equiv &\hspace{5pt}2^{4p-1} + \sqrt{4 - p} \{2^{4p-2} - (-1)2^{2p-3}2^{2p-2}p\} - 8 + p -4\sqrt{4 - p} \notag \\
\equiv &\hspace{5pt}(2^{4p-1} - 8 + p) + \sqrt{4 - p} (2^{4p-2} + 2^{4p-5}p - 4) \notag \\
\not \equiv &\hspace{5pt}0 \bmod p^2. \notag \\
 \end{split}
\end{equation}                       
Then,
$$2^{4p-1} - 2^3 + p \not \equiv 0 \bmod p^2$$
or
$$2^{4p-2} + 2^{4p-5}p - 2^2 \not \equiv 0 \bmod p^2.$$
On the other hand, the congruence relation
$$2^{2p-1} \equiv 2 - p \bmod p^2$$
is satisfied under the assumption $\lambda_p(\mathbb{Q}(\sqrt{1 - p})) > 1$. 
Substituting this in $2^{4p-1} - 2^3 + p$, we have
\begin{equation}
 \begin{split}
2^{4p-1} - 2^3 + p &\equiv 2^{2p-1}2^{2p-1}2 - 8 + p \equiv 2(2 - p)^2 - 8 + p \notag \\
                                 &\equiv 2(p^2 - 4p + 4) - 8 + p \equiv 2p^2 - 8p + 8 - 8 + p \notag \\
                                 &\equiv 2p^2 - 7p \equiv -7p \bmod p^2. \notag \\
 \end{split}
\end{equation}                       
When $p \neq 7$, we find
$$2^{4p-1} - 2^3 + p \not \equiv 0 \bmod p^2.$$
Therefore, if $p \neq 7$ and $\lambda_p(\mathbb{Q}(\sqrt{1 - p})) > 1$, then $\lambda_p(\mathbb{Q}(\sqrt{4 - p})) = 1$.
When $p = 7$, we see
$$2^{2p-1} - 2 + p = 2^{13} + 5 = 8197 \not \equiv 0 \bmod 7^2.$$
Then, $\lambda_7(\mathbb{Q}(\sqrt{1 - 7})) = \lambda_7(\mathbb{Q}(\sqrt{-6})) = 1$.
The proof of Theorem \ref{thm4} is completed.
\section{Proof of Theorem \ref{thm8}} \label{sec3} 
In this section, we show Theorem  \ref{thm8}.
The method of the proof is based on the one in \cite{By05}.
We use some properties of the Fourier coefficients of Eisenstein series of half integral weight constructed by H.~Cohen~\cite{Co75}.
The idea of the proof is used widely in the study of indivisibility of the class number of quadratic fields (cf. W.~Kohnen and K.~Ono~\cite{KO99}, K.~Ono~\cite{On99}, D.~Byeon~[2, 3, 4], I.~Kimura~\cite{Kim03}, etc).
First, we sketch the outline of the proof.
To check whether the Iwasawa $\lambda$-invariants of the cyclotomic $\mathbb{Z}_p$-extensions of imaginary quadratic fields are equal to $1$, we use the following proposition.
\begin{meidai}[cf. {[5, Proposition 2.3]}] \label{prp2.1}
Let $p$ be an odd prime number and $D$ the fundamental discriminant of an imaginary quadratic field $\mathbb{Q}(\sqrt{D})$ such that $\chi_D(p) = 1$.
Then, $L(1-p, \chi_D) / p$ is $p${\text -}integral.
Furthermore,
$$\lambda_p(\mathbb{Q}(\sqrt{D})) = 1$$
if and only if 
$$\frac{L(1-p, \chi_D)}{p} \not \equiv 0 \bmod p,$$
where $L(s, \chi_D)$ denotes the Dirichlet $L$-function.
\end{meidai}
To show this, we need some properties of the Kubota-Leopoldt $p$-adic $L$-function \cite[Lemma 1]{Wa82} and \cite[Proposition 5.1]{DFKS91}.
We can check from Proposition \ref{prp2.1} whether $\lambda_p(\mathbb{Q}(\sqrt{D})) = 1$ or not 
by studying the congruence modulo $p^2$ of the value of $L(1-p, \chi_D)$.
We can consider the value of $L(1-p, \chi_D)$ as a Fourier coefficient of Eisenstein series of half integral weight constructed by Cohen~\cite{Co75}.
Here, we state the definition of this Eisenstein series.     
     
Let $\mathfrak{r}$ be an integer greater than $1$ and $N$ a non-negative integer.
If $(-1)^{\mathfrak{r}}N \not \equiv 0, 1 \bmod 4$, then let $H(\mathfrak{r}, N) := 0$.
If $N = 0$, then let $H(\mathfrak{r}, 0) := \zeta(1 - 2\mathfrak{r}) = - B_{2\mathfrak{r}} / 2\mathfrak{r}$, where $\zeta(s)$ denotes the Riemann zeta-function and $B_{\mathfrak{r}}$ denotes the Bernoulli numbers.
If $N$ is a positive integer and $(-1)^{\mathfrak{r}}N = Dm^2$ where $D$ is the fundamental discriminant of a quadratic field $\mathbb{Q}(\sqrt{D})$ and $m$ is a positive integer, 
then define $H(\mathfrak{r}, N)$ by
$$H(\mathfrak{r}, N) := L(1 - \mathfrak{r}, \chi_D)\sum_{d | m} \mu(d)\chi_D(d)d^{\mathfrak{r}-1}\sigma_{2\mathfrak{r}-1}(m / d),$$
where $\mu (\cdot )$ denotes the M\"{o}bius function and $\sigma_{\nu}(\cdot )$ denotes the divisor function $\sigma_{\nu}(m) := \sum_{d \mid m} d^{\nu}$.
Let $M_k({\mit \Gamma}_{0} (N_1), \chi)$ denote the space of modular forms of weight $k$ on the congruence subgroup ${\mit \Gamma}_{0}(N_1)$ 
with a Dirichlet character $\chi$.
Cohen proved the following proposition.
\begin{meidai}[Cohen, \cite{Co75}] \label{prp2.2}
If $F_{\mathfrak{r}}(z) := \sum_{N = 0}^{\infty} H(\mathfrak{r}, N)q^N$,
then $F_{\mathfrak{r}}(z) \in$\\
$M_{\mathfrak{r} + \frac{1}{2}}({\mit \Gamma}_{0} (4), \chi_0)$,
where $q := e^{2\pi iz}$ and $\chi_0$ denotes the trivial character modulo $4$.
\end{meidai}
In our proof, we use this Eisenstein series.
Let $p$ be an odd prime number.
Put
$$G_p(z) := \frac{1}{p}F_p(z).$$
We see from Proposition \ref{prp2.2} that
$$G_p(z) = \sum_{N = 0}^{\infty} \frac{H(p, N)}{p}q^N \in M_{p + \frac{1}{2}}({\mit \Gamma}_{0}(4), \chi_0).$$
Assume that $N$ is a positive integer and that $(-1)^pN = Dm^2$ for some fundamental discriminant $D$ of an imaginary quadratic field such that $\chi_D(p) = 1$
 and some positive integer $m$.
If 
$$\frac{H(p, N)}{p} \not \equiv 0 \bmod p,$$
then 
$$\lambda_p(\mathbb{Q}(\sqrt{D})) = 1$$
by Proposition \ref{prp2.1}.
Considering the value of $L(1-p, \chi_D) / p$ as a Fourier coefficient of the Eisenstein series $G_p(z)$,
we can use properties of modular forms and obtain Theorem \ref{thm8}.
Next, we give a detailed proof in Section \ref{sec3.1}.
\subsection{Proof of Theorem \ref{thm8}} \label{sec3.1}
\hspace{16pt}In this section, we give a detailed proof of Theorem \ref{thm8}.
This proof relies on Sturm's theorem on congruences of modular forms \cite{St87}.
In Section \ref{sec3.1.1}, we state this theorem.
In Section \ref{sec3.1.2}, we prepare one lemma.
In Section \ref{sec3.1.3}, we prove Theorem \ref{thm8}.
\subsubsection{Sturm's theorem} \label{sec3.1.1}
In this section, we state Sturm's theorem \cite{St87}.
Let
$$g_1(z) = \sum_{N = 0}^{\infty} a_1(N)q^N$$
be any formal power series of an indeterminate $q$ with rational integer coefficients.
We define the order ord$_{p_1}(g_1)$ of $g_1(z)$ at a prime number $p_1$ by
$${\rm ord}_{p_1}(g_1) := \min\{N \mid a_1(N) \not \equiv 0\bmod p_1\}.$$
Let
$$g_2(z) = \sum_{N = 0}^{\infty} a_2(N)q^N$$
be another formal power series with rational integer coefficients and $m_2$ a rational integer.
We define $g_1(z) \equiv g_2(z) \bmod m_2$ if and only if $a_1(N) \equiv a_2(N) \bmod m_2$ for all non-negative integers $N$.
On congruences of modular forms, J.~Sturm proved the following theorem.
\begin{thm}[Sturm, \cite{St87}] \label{thm2.3}
Let $p_1$ be a prime number, $k \in \frac{1}{2}\mathbb{Z}$, $N_1$ a positive integer {\rm (}if $k \not \in \mathbb{Z}$, we assume that $4 \mid N_1)$,
and $\chi$ a Dirichlet character.  
If $g(z) \in M_k({\mit \Gamma}_{0} (N_1), \chi)$ has rational integer coefficients and
$${\rm ord}_{p_1}(g) > \frac{k}{12}[{\mit \Gamma}_{0} (1) : {\mit \Gamma}_{0} (N_1)] =  \frac{k}{12}N_1\prod_{q_2 \mid N_1, \ q_2 : {\rm prime}}(1 + q_2^{-1}),$$
then $g(z) \equiv 0 \bmod p_1$.
\end{thm}
\subsubsection{Preliminary} \label{sec3.1.2}
In this section, we prepare one lemma.
\begin{lem} \label{lem2.4}
Let $p$ be an odd prime number and 
$$\mathcal{A}_1 := \biggl\{N \in \mathbb{N}\hspace{5pt}\bigg|\hspace{5pt}\biggl(\frac{-N}{p}\biggr) = 1\biggr\}.$$
Then, there exists an integer $\alpha(p)$ coprime to $p$ such that 
$$\frac{\alpha(p)H(p, N)}{p} \in \mathbb{Z}$$
for all $N \in \mathcal{A}_1$.
\end{lem}
\begin{proof}
Assume $N \in \mathcal{A}_1$.
If $(-1)^pN = Dm^2$ where $D$ is the fundamental discriminant of a quadratic field $\mathbb{Q}(\sqrt{D})$ and $m$ is a positive integer, then
$$\frac{H(p, N)}{p} = \frac{L(1 - p, \chi_D)}{p}\sum_{d | m} \mu(d)\chi_D(d)d^{p-1}\sigma_{2p-1}(m / d)$$
by the definition of $H(p, N)$. 
Since $\sum_{d | m} \mu(d)\chi_D(d)d^{p-1}\sigma_{2p-1}(m / d)$ is an integer, 
we check the value of $L(1 - p, \chi_D) / p$.
Note that $\chi_D(p) = 1$.
Then, $L(1 - p, \chi_D) / p$ is $p$-integral by Proposition \ref{prp2.1}.
Therefore, we need to check the value of $L(1 - p, \chi_D)$.
The value of $L(1 - p, \chi_D)$ is represented by using the generalized Bernoulli number $B(p, \chi_D)$ as follows:
$$L(1 - p, \chi_D) = -\frac{B(p, \chi_D)}{p}.$$
The numbers $B(p, \chi_D) / p$ are always integers unless $D = -4$ or $D = \pm p_2$, 
in which case they have denominator $2$ or $p_2$ respectively,
where $p_2$ is an odd prime number such that $2p / (p_2 - 1)$ is an odd integer (cf. \cite{Ca59}).
Let $m_3$ be a positive integer.
If $2p / (m_3 - 1)$ is an odd integer, then $m_3 = 3$ or $m_3 = 2p + 1$.
Therefore, we can take $\alpha(p) = 6(2p + 1)$ when $p \neq 3$.
When $p = 3$, we can take $\alpha(p) = 2(2p + 1) = 14$ because of $\chi_D(p) = 1$.
The proof of Lemma \ref{lem2.4} is completed.
\end{proof}
\subsubsection{Proof of Theorem \ref{thm8}} \label{sec3.1.3}
In this section, we show Theorem \ref{thm8}.
Let $\mathfrak{S}_+ := \{p, r_1, r_2, ... , r_s\}$ and $\mathfrak{S}_- := \{r_1', r_2', ... , r_t'\}$, where $s$ and $t$ are positive integers.
Suppose $Q$ is a prime number such that $Q \not \in \mathfrak{S}_+ \cup \mathfrak{S}_-$, $\biggl(\dfrac{-D_0}{Q}\biggr) = -1$, and $Q \equiv 1 \bmod 8$.
The integer $\alpha(p)$ denotes the one in Lemma \ref{lem2.4}.
Let $p_3$ be an odd prime number and $\delta := \pm 1$.
Define $G_1(z) \in M_{p+\frac{1}{2}}({\mit \Gamma}_0(4p_3^2), \chi_0)$ by
$$G_1(z) := \alpha(p)G_p(z) \otimes \biggl(\frac{\cdot}{p_3}\biggr) = \alpha(p) \sum_{N=0}^{\infty} \biggl(\frac{N}{p_3}\biggr) \frac{H(p, N)}{p}q^N$$
and $G_2(z) \in M_{p+\frac{1}{2}}({\mit \Gamma}_0(4p_3^4), \chi_0)$ by
$$G_2(z) := \frac{1}{2} \biggl\{G_1(z) \otimes \biggl(\frac{\cdot}{p_3}\biggr) + \delta G_1(z)\biggr\} = \alpha(p) \sum_{\bigl(\frac{N}{p_3}\bigr) = \delta}^{\infty} \frac{H(p, N)}{p}q^N.$$
Repeating this process for prime numbers in $\mathfrak{S}_+ \cup \mathfrak{S}_- \cup \{Q\}$, we define $G_3(z) \in M_{p+\frac{1}{2}}({\mit \Gamma}_0(4P_1), \chi_0)$ by
$$G_3(z) := \alpha(p) \sum_{N \in \mathcal{A}_2} \frac{H(p, N)}{p}q^N,$$
where $P_1 := \prod_{r \in \mathfrak{S}_+ \cup \mathfrak{S}_- \cup \{Q\}} r^4$ and 
\begin{eqnarray}
\mathcal{A}_2 := \left\{N \in \mathbb{N}\hspace{5pt}
\begin{array}{|c}
\big(\frac{-N}{r}\big) = 1 \ {\rm for \ all} \ r \in \mathfrak{S}_+ \ {\rm and} \nonumber \\
\big(\frac{-N}{r'}\big) = -1 \ {\rm for \ all} \ r' \in \mathfrak{S}_- \cup \{Q\} \nonumber \\
\end{array}
\right \}.
\end{eqnarray}
Define $G_4(z) \in M_{p+\frac{1}{2}}({\mit \Gamma}_0(4\cdot 8^2P_1), \chi_0)$ by
$$G_4(z) := G_3(z) \otimes \chi_8 = \alpha(p) \sum_{N \equiv 3, 7 \bmod 8, \ N \in \mathcal{A}_2}^{\infty} \chi_8(N) \frac{H(p, N)}{p}q^N$$
and $G_5(z) \in M_{p+\frac{1}{2}}({\mit \Gamma}_0(4\cdot 8^4P_1), \chi_0)$ by
$$G_5(z) := \frac{1}{2} \biggl\{G_4(z) \otimes \chi_8 + \delta G_4(z)\biggr\} = \alpha(p) \sum_{N \equiv -A_1 \bmod 8, \ N \in \mathcal{A}_2}^{\infty} \frac{H(p, N)}{p}q^N,$$
where the symbol $\chi_8$ denotes the Kronecker character modulo $8$, $A_1 := 1$ if $\delta = 1$, and $A_1 := 5$ otherwise.
Let $\psi: (\mathbb{Z} / 4\mathbb{Z})^{\times} \rightarrow \mathbb{C}^{\times}$ be the Dirichlet character 
defined by $\psi(1) = 1$ and $\psi(3) = -1$.
Define $G_6(z) \in M_{p+\frac{1}{2}}({\mit \Gamma}_0(4^3P_1), \chi_0)$ by
$$G_6(z) := G_3(z) + G_3(z) \otimes \psi = \alpha(p) \sum_{N \equiv 0 \bmod 4, \ N \in \mathcal{A}_2}^{\infty} \frac{H(p, N)}{p}q^N.$$
Using $U$-operator and $V$-operator, define $G_7(z) \in M_{p+\frac{1}{2}}({\mit \Gamma}_0(4^5P_1), \chi_0)$ by
\begin{equation}
 \begin{split}
G_7(z) := &\hspace{5pt}(U_2 \mid V_2 \mid G_6)(z) - (U_4 \mid V_4 \mid G_6)(z) \notag \\
              = &\hspace{5pt}\alpha(p)\sum_{N \equiv 8\bmod 16, \ N \in \mathcal{A}_2}^{\infty} \frac{H(p, N)}{p}q^N. \notag \\
 \end{split}
\end{equation}                       
For the convenience of the reader, we write the definition of 
modular forms $(U_{p_1} \mid g)(z)$ and $(V_{p_1} \mid g)(z) \in M_k({\mit \Gamma}_{0} (p_1N_1), (\frac{{4p_1}}{\cdot}))$,
where $p_1$ is a prime number, $k \in \frac{1}{2}\mathbb{Z}$, $N_1$ a positive integer with $4 \mid N_1$, and
$g(z) := \sum_{N = 1}^{\infty} a(N)q^N \in M_k({\mit \Gamma}_{0} (N_1), \chi_0)$:
$$(U_{p_1} \mid g)(z) := \sum_{N = 1}^{\infty} a(p_1N)q^N$$
and
$$(V_{p_1} \mid g)(z) := \sum_{N = 1}^{\infty} a(N)q^{p_1N}.$$
Then, we can construct the modular form
$$G(z) := \alpha(p) \sum_{N \equiv -A \bmod B, \ N \in \mathcal{A}_2}^{\infty} \frac{H(p, N)}{p}q^N \in M_{p+\frac{1}{2}}({\mit \Gamma}_0(P_2), \chi_0),$$
where $(A, B)$ is a pair of fixed integers in $\{(1, 8), (5, 8), (8, 16)\}$, $P_2 := 4^5P_1$ if $(A, B) = (8, 16)$, and $P_2 := 4^7P_1$ otherwise (see the definition of $G_5(z)$ and $G_7(z)$).
If $(-1)^pN = -N = Dm^2$ for some fundamental discriminant $D$ of an imaginary quadratic field and some positive integer $m$, it follows that 
$D \equiv A \bmod B$ from $N \equiv -A \bmod B$ and that 
$$\biggl(\frac{D}{p}\biggr) = \biggl(\frac{D}{r_i}\biggr) = 1 \ {\rm and} \ \biggl(\frac{D}{Q}\biggr) = \biggl(\frac{D}{r'_j}\biggr) = -1$$
from
$$\biggl(\frac{-N}{p}\biggr) = \biggl(\frac{-N}{r_i}\biggr) = 1 \ {\rm and} \ \biggl(\frac{-N}{Q}\biggr) = \biggl(\frac{-N}{r'_j}\biggr) = -1,$$
where $i$ runs over $1, 2, ..., s$ and $j$ runs over $1, 2, ..., t$.
Put
$$\kappa := 1 + 2^{m_4}(2p + 1) \prod_{r \in \mathfrak{S}_+ \cup \mathfrak{S}_- \cup \{Q\}} r^3(r + 1),$$
where $m_4 := 10$ if $A = 1, 5$ and $m_4 := 6$ otherwise.
Put
$$\mathcal{A}_3 := \{N \in \mathcal{A}_2 \mid 1 \le N \le \kappa \ {\rm and} \ -N \equiv A \bmod B\}$$
and
$$\mathfrak{S}_p := \{D_N: \ {\rm fundamental \ discriminant \ of} \ \mathbb{Q}(\sqrt{-N}) \mid N \in \mathcal{A}_3\}.$$
It is essential for the proof of Theorem \ref{thm8} to show the following theorem.
\begin{thm} \label{thm2.5}
Let $l$ be an odd prime number satisfying the following conditions:
\begin{eqnarray*}
 \left\{
 \begin{array}{l}
{\rm (i)} \ \kappa < l,\\
{\rm (ii)} \ l \equiv 1 \bmod P_3,\\
{\rm (iii)} \ \chi_D(l) = -1\hspace{10pt}for \ all \ D \in \mathfrak{S}_p \cup \{D_0\},\\
\end{array}
\right.\\
\end{eqnarray*}
where $P_3 := B\prod_{r \in \mathfrak{S}_+ \cup \mathfrak{S}_- \cup \{Q\}} r$.
Then, there exists a positive integer $N_l$ satisfying the following conditions:
\begin{eqnarray*}
 \left\{
 \begin{array}{l}
{\rm (I)} \ 1 \le N_l \le l\kappa,\\
{\rm (II)} \ l \nmid N_l,\\
{\rm (III)} \ lN_l \equiv -A \bmod B,\\
{\rm (IV)} \ lN_l \in \mathcal{A}_2,\\
{\rm (V)} \ \dfrac{H(p, lN_l)}{p} \not \equiv 0 \bmod p.\\
\end{array}
\right.\\
\end{eqnarray*}
\end{thm}
We prove the existence of prime numbers $l$ satisfying the above conditions later (see Lemma \ref{lem2.8}).\\
~\\
{\bf Proof of Theorem \ref{thm2.5}.} \ By using $U$-operator and $V$-operator, we construct the following modular form
$$G_l(z) := (U_l \mid G)(z) - 3(V_l \mid G)(z) \in M_{p + \frac{1}{2}}\biggl({\mit \Gamma}_{0} (lP_2), \biggl(\frac{{4l}}{\cdot}\biggl)\biggr).$$
The coefficient $b(N)$ of $q^N$ in $G_l(z)$ is represented as follows:
$$b(N) = \frac{\alpha(p)}{p}\biggl\{H(p, lN) - 3H\biggl(p, \frac{N}{l}\biggr)\biggr\}.$$
First, we treat the case where $l \mid N$.
We can write $N = lN'$ for some integer $N'$.
Then,
$$b(N) = \frac{\alpha(p)}{p}\{H(p, l^2N') - 3H(p, N')\}.$$
If $N' \not \equiv -A \bmod B$ or $N' \not \in \mathcal{A}_2$, then $b(N) = 0$.
If $N'$ satisfies the conditions 
$N' \equiv -A \bmod B$, $N' \in \mathcal{A}_2$, and $1 \le N' \le \kappa$,
we can show $b(N) \equiv 0 \bmod p$ as follows.
We write $-N' = D_{N'}m'^{2}$ for some fundamental discriminant $D_{N'}$ of an imaginary quadratic field and some positive integer $m'$.
Then,
$$b(N) = \frac{\alpha(p)}{p}\{H(p, -D_{N'}(lm')^2) - 3H(p, -D_{N'}m'^2)\}.$$
We see from the definition of $H(p, N)$ that
\begin{equation} \label{eq7}
 \begin{split}
   &\frac{\alpha(p)H(p, -D_{N'}(lm')^2)}{p} \\
= &\frac{\alpha(p)H(p, -D_{N'})}{p}\biggl\{\sum_{d \mid lm'} \mu(d)\chi_{D_{N'}}(d)d^{p-1}\sigma_{2p-1}(lm' / d)\biggr\}. \\
 \end{split}
\end{equation}                       
Since $l > \kappa$ and $\kappa \ge N'$ hold, we have $\gcd(l, m') = 1$.
Then,
\begin{equation}
 \begin{split}
    &\sum_{d \mid lm'} \mu(d)\chi_{D_{N'}}(d)d^{p-1}\sigma_{2p-1}(lm' / d) \notag \\
= &\sum_{d \mid m'} \bigl\{\mu(d)\chi_{D_{N'}}(d)d^{p-1}\sigma_{2p-1}(lm' / d) + \mu(ld)\chi_{D_{N'}}(ld)l^{p-1}d^{p-1}\sigma_{2p-1}(m' / d) \bigr\} \notag \\ 
= &\sum_{d \mid m'} \bigl\{(1 + l^{2p-1})\mu(d)\chi_{D_{N'}}(d)d^{p-1}\sigma_{2p-1}(m' / d)  \notag \\ 
   &- \mu(d)\chi_{D_{N'}}(l)\chi_{D_{N'}}(d)l^{p-1}d^{p-1}\sigma_{2p-1}(m' / d) \bigr\} \notag \\ 
= & \ (1 + l^{2p-1} + l^{p-1}) \sum_{d \mid m'} \mu(d)\chi_{D_{N'}}(d)d^{p-1}\sigma_{2p-1}(m' / d). \notag \\   
 \end{split}
\end{equation}                       
Substituting this in equation (\ref{eq7}), we obtain
$$\frac{\alpha(p)H(p, -D_{N'}(lm')^2)}{p} = (1 + l^{2p-1} + l^{p-1})\frac{\alpha(p)H(p, -D_{N'}m'^2)}{p}.$$
We see from the assumption $l \equiv 1 \bmod p$ that
$$(1 + l^{2p-1} + l^{p-1})\frac{\alpha(p)H(p, -D_{N'}m'^2)}{p} \equiv \frac{3\alpha(p)H(p, -D_{N'}m'^2)}{p} \bmod p.$$
Then, 
$$b(N) \equiv 0 \bmod p.$$
Therefore, if $l \mid N$ and $l \le N \le l\kappa$ hold, we have $b(N) \equiv 0 \bmod p$.
Next, we treat the case where $l \nmid N$.
In this case, $H(p, N / l) = 0$.
Then,
$$b(N) = \frac{\alpha(p)H(p, lN)}{p}.$$
Assume that 
$$\frac{H(p, lN)}{p} \equiv 0 \bmod p$$
for all positive integers $N$ such that
\begin{eqnarray*}
 \left\{
 \begin{array}{l}
{\rm (I)} \ 1 \le N \le l\kappa,\\
{\rm (II)} \ l \nmid N,\\
{\rm (III)} \ lN \equiv -A \bmod B,\\
{\rm (IV)} \ lN \in \mathcal{A}_2.\\
\end{array}
\right.\\
\end{eqnarray*}
Then, $b(N) \equiv 0 \mod p$ for all integers $N$ such that $1 \le N \le l\kappa$.
Note that $G_l(z) \in \mathbb{Z}[\![q]\!]$ by Lemma \ref{lem2.4}.
Therefore, ${\rm ord}_p(G_l(z)) > l\kappa$.
Here, we use Theorem \ref{thm2.3}.
We see
\begin{equation}
 \begin{split}
[{\mit \Gamma}_{0} (1) : {\mit \Gamma}_{0} (lP_2)] &= lP_2 \prod_{q_2 \mid lP_2, \ q_2: prime} (1 + q_2^{-1}) \notag \\   
&= l\cdot 4^{m_5} \prod_{r \in \mathfrak{S}_+ \cup \mathfrak{S}_- \cup \{Q\}} r^4 \prod_{r \in \mathfrak{S}_+ \cup \mathfrak{S}_- \cup \{2, l, Q\}} (1 + r^{-1}) \notag \\   
&= \frac{3\cdot 4^{m_5}l(l + 1)}{2l} \prod_{r \in \mathfrak{S}_+ \cup \mathfrak{S}_- \cup \{Q\}} r^4 
\prod_{r \in \mathfrak{S}_+ \cup \mathfrak{S}_- \cup \{Q\}} \frac{r + 1}{r} \notag \\                                                                                                             
&= \frac{3\cdot 4^{m_5}(l + 1)}{2} \prod_{r \in \mathfrak{S}_+ \cup \mathfrak{S}_- \cup \{Q\}} r^3(r + 1), \notag \\  
 \end{split}
\end{equation}                       
where $m_5 := 5$ if $(A, B) = (8, 16)$ and $m_5 := 7$ otherwise.
Then, 
\begin{equation}
 \begin{split}
\frac{1}{12}\biggl(p + \frac{1}{2}\biggr)[{\mit \Gamma}_{0} (1) : {\mit \Gamma}_{0} (lP_2)] 
&= \frac{3\cdot 4^{m_5}(l + 1)(2p + 1)}{48} \prod_{r \in \mathfrak{S}_+ \cup \mathfrak{S}_- \cup \{Q\}} r^3(r + 1) \notag \\   
&= 4^{m_5-2}(l + 1)(2p + 1) \prod_{r \in \mathfrak{S}_+ \cup \mathfrak{S}_- \cup \{Q\}} r^3(r + 1) \notag \\   
&= (l + 1)(\kappa - 1) \notag \\                                                                                                                                                                
&= l\kappa - l + \kappa - 1. \notag \\                                                                                                                                                                 
 \end{split}
\end{equation}                       
Since $\kappa < l$ holds, we obtain
$$\frac{1}{12}\biggl(p + \frac{1}{2}\biggr)[{\mit \Gamma}_{0} (1) : {\mit \Gamma}_{0} (lP_2)] < l\kappa.$$
Then,
$$\frac{1}{12}\biggl(p + \frac{1}{2}\biggr)[{\mit \Gamma}_{0} (1) : {\mit \Gamma}_{0} (lP_2)] < {\rm ord}_p(G_l(z)).$$
Hence it follows from Theorem \ref{thm2.3} that 
\begin{equation} 
G_l(z) \equiv 0 \bmod p. \nonumber \\
\end{equation}
On the other hand,
\begin{equation} \label{eq8}
b(-D_0l^3) \not \equiv 0 \bmod p,
\end{equation}
a contradiction.
Equation (\ref{eq8}) is given as follows.
\begin{equation}
 \begin{split}
 b(-D_0l^3) = & \ \frac{\alpha(p)}{p}\{H(p, -D_0l^4) - 3H(p, -D_0l^2)\} \notag \\   
= & \ \frac{\alpha(p)H(p, -D_0)}{p}\biggl\{\sum_{d \mid l^2} \mu(d)\chi_{D_{0}}(d)d^{p-1}\sigma_{2p-1}(l^2 / d) \notag \\
   & \ - 3\sum_{d \mid l} \mu(d)\chi_{D_{0}}(d)d^{p-1}\sigma_{2p-1}(l / d)\biggr\} \notag \\
= & \ \frac{\alpha(p)H(p, -D_0)}{p}\bigl\{\sigma_{2p-1}(l^2) + \mu(l)\chi_{D_{0}}(l)l^{p-1}\sigma_{2p-1}(l) \notag \\
   & \ - 3\bigl(\sigma_{2p-1}(l) + \mu(l)\chi_{D_{0}}(l)l^{p-1}\sigma_{2p-1}(1)\bigr)\bigr\} \notag \\
= & \ \frac{\alpha(p)H(p, -D_0)}{p}\bigl\{1 + l^{2p-1} + l^{2(2p-1)} + l^{p-1}(1 + l^{2p-1}) \notag \\
   & \ - 3\bigl(1 + l^{2p-1} + l^{p-1}\bigr)\bigr\} \notag \\
\equiv & \ \frac{(3 + 2 - 9)\alpha(p)H(p, -D_0)}{p} \notag \\
\equiv & \ \frac{-4\alpha(p)H(p, -D_0)}{p} \notag \\
\not \equiv & \ 0 \bmod p. \notag \\ 
  \end{split}
\end{equation}                       
The congruence relation $\dfrac{\alpha(p)H(p, -D_0)}{p} \not \equiv 0 \bmod p$ follows
 from the assumption $\lambda_p(\mathbb{Q}(\sqrt{D_0})) = 1$ and Proposition \ref{prp2.1}.
Then, there exists a positive integer $N$ such that 
\begin{eqnarray*}
 \left\{
 \begin{array}{l}
{\rm (I)} \ 1 \le N \le l\kappa,\\
{\rm (II)} \ l \nmid N,\\
{\rm (III)} \ lN \equiv -A \bmod B,\\
{\rm (IV)} \ lN \in \mathcal{A}_2,\\
{\rm (V)} \ \dfrac{H(p, lN)}{p} \not \equiv 0 \bmod p.\\
\end{array}
\right.\\
\end{eqnarray*}
We can take the above $N$ as $N_l$.
The proof of Theorem \ref{thm2.5} is completed.
\begin{rem}
We need the assumption $D_0 \neq -8$ in the above proof.
Since $Q \equiv 1 \bmod 8$ holds, we see
$$\biggl(\frac{8}{Q}\biggr) = \biggl(\frac{2}{Q}\biggr)^3 = 1 \neq -1.$$
Then, the coefficient $b(8l^3)$ in $G_l(z)$ is zero.
Therefore, if $D_0 = -8$, we can not use the above discussion.
\end{rem}
Using Theorem \ref{thm2.5}, we show Theorem \ref{thm8}.
Let $l$ and $N_l$ be integers as in Theorem \ref{thm2.5}. 
Then, the fundamental discriminant $D_l$ of $\mathbb{Q}(\sqrt{-lN_l})$ satisfies $\lambda_p(\mathbb{Q}(\sqrt{D_l})) = 1$ by Proposition \ref{prp2.1}.
Moreover, $D_l$ satisfies $D_l \equiv A \bmod B$, $\chi_r(D_l) = 1$ for all $r \in \mathfrak{S}_+$, and $\chi_{r'}(D_l) = -1$ for all $r' \in \mathfrak{S}_-$. 
Therefore, we can prove Theorem \ref{thm8} by estimating the number of imaginary quadratic fields $\mathbb{Q}(\sqrt{-lN_l})$.
First, we check the existence of the prime number $l$.
\begin{lem} \label{lem2.8}
There exist infinitely many prime numbers $l$ satisfying the following conditions:
\begin{eqnarray*}
 \left\{
 \begin{array}{l}
{\rm (i)} \ \kappa < l,\\
{\rm (ii)} \ l \equiv 1 \bmod P_3,\\
{\rm (iii)} \ \chi_D(l) = -1\hspace{10pt}for \ all \ D \in \mathfrak{S}_p \cup \{D_0\}.\\
\end{array}
\right.\\
\end{eqnarray*}
\end{lem}
\begin{proof}
We use the Chebotarev density theorem.
For the convenience, we rewrite $\mathfrak{S}_p = \{D_1, D_2, ..., D_{n_1}\}$, where $n_1$ is a positive integer.
Put
$$\mathfrak{A} := \mathbb{Q}\bigl(\zeta_{P_3}, \sqrt{D_0D_1}, \sqrt{D_0D_2}, ..., \sqrt{D_0D_{n_1}}\bigr)$$
and 
$$\mathfrak{B} := \mathfrak{A}(\sqrt{D_0}).$$
Let $\mathbb{Q}(\zeta_{v_0})$ be a cyclotomic field containing $\mathfrak{B}$.\\
\hspace{17pt}First, we will show $\sqrt{D_0} \not \in \mathfrak{A}$, that is, $[\mathfrak{B} : \mathfrak{A}] = 2$ as follows.\\
(1)  We treat the case where $D \equiv 1 \bmod 8$ (resp. $D \equiv 5 \bmod 8$) for all $D \in \mathfrak{S}_p \cup \{D_0\}$, 
that is, $(A, B) = (1, 8)$ (resp. $(A, B) = (5, 8)$).
Since $r \nmid D_0D_1 \cdots D_{n_1}$ for all $r \in \mathfrak{S}_+ \cup \mathfrak{S}_- \cup \{2, Q\}$,
it is sufficient to show 
$$\sqrt{D_0} \not \in \mathfrak{C} := \mathbb{Q}(\sqrt{-1}, \sqrt{D_0D_1}, \sqrt{D_0D_2}, ..., \sqrt{D_0D_{n_1}}).$$
Suppose $\sqrt{D_0} \in \mathfrak{C}$.
We can write 
$$D_0 = -\frac{(D_0D_1)^{\varepsilon_1}(D_0D_2)^{\varepsilon_2}\cdots(D_0D_{n_1})^{\varepsilon_{n_1}}}{\square},$$
where $\varepsilon_i \in \{0, 1\}$ for each $i \in \{1, 2, ... , n_1\}$ and $\square$ denotes the square part of 
$(D_0D_1)^{\varepsilon_1}(D_0D_2)^{\varepsilon_2}\cdots(D_0D_{n_1})^{\varepsilon_{n_1}}$.
Since $D \equiv 1 \bmod 4$ for all $D \in \mathfrak{S}_p \cup \{D_0\}$ and $\square \equiv 1 \bmod 4$ holds, we have
$$-\frac{(D_0D_1)^{\varepsilon_1}(D_0D_2)^{\varepsilon_2}\cdots(D_0D_{n_1})^{\varepsilon_{n_1}}}{\square} \equiv 3 \bmod 4.$$ 
On the other hand, 
$$D_0 \equiv 1 \bmod 4.$$
This is a contradiction.
Then, $\sqrt{D_0} \not \in \mathfrak{C}$, that is, $\sqrt{D_0} \not \in \mathfrak{A}$.\\
(2) \ We treat the case where $D \equiv 8 \bmod 16$ for all $D \in \mathfrak{S}_p \cup \{D_0\}$, that is, $(A, B) = (8, 16)$.
Since $r \nmid D_0D_1 \cdots D_{n_1}$ for all $r \in \mathfrak{S}_+ \cup \mathfrak{S}_- \cup \{Q\}$,
it is sufficient to show 
$$\sqrt{D_0} \not \in \mathfrak{D} := \mathbb{Q}(\sqrt{-1}, \sqrt{2}, \sqrt{D_0D_1}, \sqrt{D_0D_2}, ..., \sqrt{D_0D_{n_1}}).$$
Suppose $\sqrt{D_0} \in \mathfrak{D}$.
We can write 
$$D_0 = -8 \cdot \frac{(D_0D_1)^{\varepsilon_1}(D_0D_2)^{\varepsilon_2}\cdots(D_0D_{n_1})^{\varepsilon_{n_1}}}{\square},$$
where $\varepsilon_i \in \{0, 1\}$ for each $i \in \{1, 2, ... , n_1\}$ and $\square$ denotes the square part of 
$c_1 := (D_0D_1)^{\varepsilon_1}(D_0D_2)^{\varepsilon_2}\cdots(D_0D_{n_1})^{\varepsilon_{n_1}}$.
When $\varepsilon_i = 0$ for all $i \in \{1, 2, ... , n_1\}$, we see $D_0 = -8$, a contradiction.
Therefore, we treat the case where there exists $i \in \{1, 2, ... , n_1\}$ such that $\varepsilon_i = 1$.
Since
$$\biggl(\frac{D_0D_i}{Q}\biggr) = (-1)^2 = 1$$
for all $i \in \{1, 2, ... , n_1\}$ and $\biggl(\dfrac{-8}{Q}\biggr) = 1$ holds, 
we have
$$\biggl(\dfrac{-8(c_1 / \square)}{Q}\biggr) = \biggl(\dfrac{-8c_1}{Q}\biggr) = \biggl(\dfrac{c_1}{Q}\biggr) = 1.$$
On the other hand, $\biggl(\dfrac{D_0}{Q}\biggr) = -1$. 
This is a contradiction.
Then, $\sqrt{D_0} \not \in \mathfrak{D}$, that is, $\sqrt{D_0} \not \in \mathfrak{A}$.\\
\hspace{17pt}It follows from $[\mathfrak{B} : \mathfrak{A}] = 2$ that there exists
$$\tau \in {\rm Gal}(\mathbb{Q}(\zeta_{v_0}) / \mathfrak{A}) \smallsetminus {\rm Gal}(\mathbb{Q}(\zeta_{v_0}) / \mathfrak{B}).$$
We see from the Chebotarev density theorem that
there exist infinitely many prime numbers $l$ 
such that there is a prime ideal $\wp$ of $\mathbb{Q}(\zeta_{v_0})$ over $l$ 
with $\tau = \biggl[\dfrac{\mathbb{Q}(\zeta_{v_0}) / \mathbb{Q}}{\wp}\biggr]$,
where $\biggl[\dfrac{\mathbb{Q}(\zeta_{v_0}) / \mathbb{Q}}{\wp}\biggr]$ is the Frobenius of $\wp$ in ${\rm Gal}(\mathbb{Q}(\zeta_{v_0}) / \mathbb{Q})$.
Since the number of primes $l'$ such that $l' \le \kappa$, the number of prime factors of $D_0$, and the number of prime factors of $v_0$ are finite,
there exist infinitely many prime numbers $l$ with $\kappa < l$ and $l \nmid v_0D_0$ such that 
there is a prime ideal $\wp$ of $\mathbb{Q}(\zeta_{v_0})$ over $l$ 
with $\tau = \biggl[\dfrac{\mathbb{Q}(\zeta_{v_0}) / \mathbb{Q}}{\wp}\biggr]$.\\
\hspace{16pt}Next, we will show that such prime numbers $l$ satisfy the conditions (ii) and (iii) of this lemma as follows.
Note that such prime number $l$ is unramified in $\mathbb{Q}(\zeta_{v_0})$, that is, the prime ideal $\wp$ is unramified in $\mathbb{Q}(\zeta_{v_0})$.
Then, the decomposition group of $\wp$ in ${\rm Gal}(\mathbb{Q}(\zeta_{v_0}) / \mathbb{Q})$ is $\langle \tau \rangle$.
Since ${\rm Gal}(\mathbb{Q}(\zeta_{v_0}) / \mathfrak{A})$ contains $\langle \tau \rangle$,
the subfield of $\mathbb{Q}(\zeta_{v_0})$ corresponded with this decomposition group contains $\mathfrak{A}$.
Then, $l$ splits completely in $\mathfrak{A}$.
Since $\mathbb{Q}(\zeta_{P_3})$ is contained in $\mathfrak{A}$, the prime number $l$ splits completely in $\mathbb{Q}(\zeta_{P_3})$.
The property that $l$ splits completely in $\mathbb{Q}(\zeta_{P_3})$ is equivalent to $l \equiv 1 \bmod P_3$ (see \cite[Theorem 2.13]{Wa}).
Then, $l$ satisfies the condition (ii) of this lemma.
It follows from $\mathbb{Q}(\sqrt{D_0D_i}) \subset \mathfrak{A}$ that 
the prime number $l$ splits completely in $\mathbb{Q}(\sqrt{D_0D_i})$, where $i \in \{1, 2, ... , n_1\}$.
On the other hand, we see from $\tau \not \in {\rm Gal}(\mathbb{Q}(\zeta_{v_0}) / \mathfrak{B})$ that 
the prime ideals of $\mathfrak{A}$ over $l$ do not split completely in $\mathfrak{B}$.
Then, $l$ does not split completely in $\mathbb{Q}(\sqrt{D_0})$.
We see from $l \nmid D_0$ that $l$ is inert in $\mathbb{Q}(\sqrt{D_0})$. 
Therefore, we find that $\chi_{D}(l) = -1$ for all $D \in \mathfrak{S}_p \cup \{D_0\}$.
The prime number $l$ satisfies the condition (iii) of this lemma.
The proof of Lemma \ref{lem2.8} is completed.
\end{proof}
Next, we estimate the number of imaginary quadratic fields $\mathbb{Q}(\sqrt{-lN_l})$.
In Lemma \ref{lem2.8}, we proved that there exist infinitely many prime numbers $l$ satisfying the conditions (i) $\kappa < l$, 
(ii) $l \equiv 1 \bmod P_3$, and (iii) $\chi_D(l) = -1$ for all $D \in \mathfrak{S}_p \cup \{D_0\}$.
Since the set of prime numbers $l$ such that $l \equiv 1 \bmod P_3$ and $\chi_D(l) = -1$ for all $D \in \mathfrak{S}_p \cup \{D_0\}$ is not empty,
we see from the Chinese remainder theorem that 
there is an arithmetic progression $u_1$ modulo $v_1$ with $\gcd(u_1, v_1) = 1$ 
which contains infinitely many such prime numbers $l$.
Here, we need the following lemma.
\begin{lem} \label{lem2.9}
Let 
\begin{eqnarray}
\mathcal{A}_4 := \left\{l: \ odd \ prime\hspace{5pt}
\begin{array}{|c}
l \equiv u_1 \bmod v_1 \ and \ l > \kappa \nonumber \\
\end{array}
\right \},
\end{eqnarray}
where $u_1$ and $v_1$ are the integers mentioned in the above. 
Suppose $D_l$ is the fundamental discriminant of an imaginary quadratic field $\mathbb{Q}(\sqrt{-lN_l})$,
where $l$ and $N_l$ denote the integers in Theorem \ref{thm2.5}.
Then, for any fixed $l \in \mathcal{A}_4$,
there exists at most one element $l' \in \mathcal{A}_4$ such that $D_l = D_{l'}$. 
\end{lem}
\begin{proof}
For the convenience, we rewrite $\mathcal{A}_4 = \{l_i \mid i \in \mathbb{N}\}$.
For $l_i$, $l_j \in \mathcal{A}_4$, we assume that $l_i < l_j$ if $i$ and $j$ are positive integers with $i < j$.
Suppose $D_{l_h} = D_{l_i} = D_{l_j}$, where $h$, $i$, and $j$ are positive integers with $h < i < j$.
Then, 
$$-\frac{l_hN_{l_h}}{\square} = -\frac{l_iN_{l_i}}{\square} = -\frac{l_jN_{l_j}}{\square},$$
where $\square$ is the square part of the numerator.
By Theorem \ref{thm2.5},
$$\gcd(l_h, N_{l_h}) = \gcd(l_i, N_{l_i}) = \gcd(l_j, N_{l_j}) = 1.$$
This implies that
$$l_i l_j \mid N_{l_h}.$$
Then,
$$l_i l_j \le N_{l_h}.$$
On the other hand, we see from Theorem \ref{thm2.5} that 
$N_{l_h} \le l_h\kappa$ and $\kappa < l_j$.
Then, 
$$N_{l_h} < l_i l_j.$$
This is a contradiction.
The proof of Lemma \ref{lem2.9} is completed.
\end{proof}
From $0 < -D_l \le lN_l \le l^2\kappa$ and Lemma \ref{lem2.9},
we obtain the following:
\begin{eqnarray}
\sharp \left\{D \in S_-(X)\hspace{5pt}
\begin{array}{|c}
\lambda_p(\mathbb{Q}(\sqrt{D})) = 1, \nonumber \\
D \equiv A \bmod B, \nonumber \\
\big(\frac{D}{r}\big) = 1 \ {\rm for \ all} \ r \in \mathfrak{S}_+, \ {\rm and} \nonumber \\
\big(\frac{D}{r'}\big) = -1 \ {\rm for \ all} \ r' \in \mathfrak{S}_- \nonumber \\
\end{array}
\right \}
\end{eqnarray}
\begin{eqnarray}
\hspace{-37pt}\ge \sharp\left\{D_l \in S_-(X)\hspace{5pt}
\begin{array}{|c}
\mathbb{Q}(\sqrt{D_l}) = \mathbb{Q}(\sqrt{-lN_l}), \nonumber \\ 
{\rm where} \ l \in \mathcal{A}_4 \ {\rm and} \nonumber \\ 
N_l \in \mathbb{N} \ {\rm in \ Theorem \ \ref{thm2.5}} \nonumber \\
\end{array}
\right \} 
\end{eqnarray}
\begin{equation}
 \begin{split}
\ge &\frac{1}{2}\sharp \{ \ l \in \mathcal{A}_4 \mid 0 \le l^2\kappa \le X\} \notag \\ 
= &\frac{1}{2}\sharp \{ \ l: {\rm odd \ prime} \mid l \equiv u_1 \bmod v_1, \ l > \kappa, \ {\rm and} \ l^2\kappa \le X\} \notag \\    
 \end{split}
\end{equation} 
\begin{equation}
 \begin{split}                         
= &\hspace{5pt}\frac{1}{2}\sharp \biggl\{ \ l: {\rm odd \ prime}\hspace{5pt}\bigg|\hspace{5pt}l \equiv u_1 \bmod v_1 \ {\rm and} \ 0 < l \le \dfrac{\sqrt{X}}{\sqrt{\kappa}}\biggr\} \notag \\                                
&-\frac{1}{2}\sharp \{ \ l: {\rm odd \ prime}\hspace{5pt}|\hspace{5pt}l \equiv u_1 \bmod v_1 \ {\rm and} \ 0 < l \le \kappa \} \notag \\ 
 \end{split}
\end{equation} 
for any sufficiently large $X \in \mathbb{R}$.
Let
$$E(X) := \frac{1}{2}\sharp \biggl\{ \ l: {\rm odd \ prime}\hspace{5pt}\bigg|\hspace{5pt}l \equiv u_1 \bmod v_1 \ {\rm and} \ 0 < l \le \dfrac{\sqrt{X}}{\sqrt{\kappa}}\biggr\}.$$
By the Dirichlet's theorem on primes in arithmetic progression, we have
$$E(X) \sim \frac{\sqrt{X}}{\varphi(v_1)\sqrt{\kappa}(\log X - \log \kappa)}$$
for any sufficiently large $X \in \mathbb{R}$.
Then, under the assumption of the existence of $D_0$, we see that
\begin{eqnarray}
\sharp \left\{D \in S_-(X)\hspace{5pt}
\begin{array}{|c}
\lambda_p(\mathbb{Q}(\sqrt{D})) = 1, \nonumber \\
D \equiv A \bmod B, \nonumber \\
\big(\frac{D}{r}\big) = 1 \ {\rm for \ all} \ r \in \mathfrak{S}_+, \ {\rm and} \nonumber \\
\big(\frac{D}{r'}\big) = -1 \ {\rm for \ all} \ r' \in \mathfrak{S}_- \nonumber \\
\end{array}
\right \} \gg \frac{\sqrt{X}}{\log X}
\end{eqnarray}
for any sufficiently large $X \in \mathbb{R}$.
The proof of Theorem \ref{thm8} is completed.
\section{Proofs of Theorems \ref{thm14}, \ref{thm14.1}, \ref{thm12}, and \ref{thm12.1}} \label{sec4} 
In this section, we prove Theorems \ref{thm14}, \ref{thm14.1}, \ref{thm12}, and \ref{thm12.1}.
To show Theorems \ref{thm14} and \ref{thm12}, we will prove the following theorem.
\begin{thm} \label{thm3.1}
Let $p$ be an odd prime number, $n$ an integer greater than $1$ such that $\gcd(p, n) = 1$, and $x_1$ a positive integer such that $\gcd(p, x_1) = 1$.\\
$(1)$ \ Assume $x_1^2 < p^n$.
Then, $\lambda_p(\mathbb{Q}(\sqrt{x_1^2 - p^n})) = 1$ if and only if $p \nmid h(\mathbb{Q}(\sqrt{x_1^2 - p^n}))$ and $(2x_1)^{p-1} \not \equiv 1 \bmod p^2$.\\
$(2)$ \ Assume $x_1^2 < 4p^n$ and $x_1$ is odd.
Then, $\lambda_p(\mathbb{Q}(\sqrt{x_1^2 - 4p^n})) = 1$ if and only if $p \nmid h(\mathbb{Q}(\sqrt{x_1^2 - 4p^n}))$ and $x_1^{p-1} \not \equiv 1 \bmod p^2$.
\end{thm}
Applying Theorem \ref{thm3.1} (2) to $x_1 = 1$, we obtain the imaginary quadratic fields treated in Theorem \ref{thm14}.
Applying Theorem \ref{thm3.1} (1) to $x_1 = 1$, $2$, $q_1$, $2q_1$, we obtain the imaginary quadratic fields treated in Theorem \ref{thm12}.
Here, we give some additional explanation for the case where $x_1 = 2$.
Note that $2^{p-1} \not \equiv -1 \bmod p^2$.
In fact, if $2^{p-1} \equiv -1 \bmod p^2$, the congruence relation $2^{p-1} + 1 \equiv 0 \bmod p$ and $2^{p-1} - 1 \equiv 0 \bmod p$ hold.
On the other hand, $\gcd(2^{p-1} + 1, 2^{p-1} - 1) = 1$, a contradiction. 
Therefore, if $2^{p-1} \not \equiv 1 \bmod p^2$, then $2^{2(p-1)} \not \equiv 1 \bmod p^2$.
Thus, using Theorem \ref{thm3.1} (1) for the case where $x_1 = 2$,
we can find that if $p \nmid h(\mathbb{Q}(\sqrt{4 - p^n}))$ and $2^{p-1} \not \equiv 1 \bmod p^2$, then $\lambda_p(\mathbb{Q}(\sqrt{4 - p^n})) = 1$.    

To show Theorems \ref{thm14.1} and \ref{thm12.1}, we will prove the following theorem.
\begin{thm} \label{thm3.2}
Let $p$ be an odd prime number, $q_1$ a prime factor of $p - 2$ such that $q_1^{p-1} \not \equiv 1 \bmod p^2$, and $n$ an integer greater than $1$.\\
$(1)$ \ Assume $p \equiv 3\bmod 4$.\\
\hspace{10pt}{\rm (i)} \ Suppose $p \neq 3$.
If $n$ is an odd integer, then the order of the ideal class containing the prime ideal of the imaginary quadratic field $\mathbb{Q}(\sqrt{1 - p^n})$ over $p$ is $n$.\\
\hspace{10pt}{\rm (ii)} \ Suppose $p = 3$.
If $n$ is an odd integer such that $n \neq 5$, 
then the order of the ideal class containing the prime ideal of the imaginary quadratic field $\mathbb{Q}(\sqrt{1 - p^n})$ over $p$ is $n$.
When $n = 5$, we see that $\mathbb{Q}(\sqrt{1 - 3^5}) = \mathbb{Q}(\sqrt{-2})$.
The class number of $\mathbb{Q}(\sqrt{-2})$ is $1$ 
and the order of the ideal class containing the prime ideal of $\mathbb{Q}(\sqrt{-2})$ over $3$ is $1$.\\
\hspace{10pt}{\rm (iii)} \ Suppose $2^{p-1} \equiv 1 \bmod p^2$.
If $n$ is an odd composite number, then the order of the ideal class containing the prime ideal of the imaginary quadratic field 
$\mathbb{Q}(\sqrt{q_1^2 - p^n})$ over $p$ is $n$.\\   
$(2)$ \ Assume $p \equiv 1\bmod 4$.\\
\hspace{10pt}{\rm (i)} \ Suppose $\mathbb{Q}(\sqrt{4 - p^n}) \neq \mathbb{Q}(\sqrt{-1})$.
Then, the order of the ideal class containing the prime ideal of the imaginary quadratic field 
$\mathbb{Q}(\sqrt{4 - p^n})$ over $p$ is $n$ (see \cite[Theorem 1 (2)]{It11}).\\
\hspace{10pt}{\rm (ii)} \ Suppose $2^{p-1} \equiv 1 \bmod p^2$ and $\mathbb{Q}(\sqrt{4q_1^2 - p^n}) \neq \mathbb{Q}(\sqrt{-1})$.
If $n$ is a composite number, then the order of the ideal class containing the prime ideal of the imaginary quadratic field 
$\mathbb{Q}(\sqrt{4q_1^2 - p^n})$ over $p$ is $n$.\\
$(3)$ \ If $n$ is greater than $8$, 
then the order of the ideal class containing the prime ideal of the imaginary quadratic field 
$\mathbb{Q}(\sqrt{1 - 4p^n})$ over $p$ is $n$ (see \cite[Theorem]{Is} or \cite{It}).
\end{thm}
Theorems \ref{thm14.1} and \ref{thm12.1} follow from Theorem \ref{thm3.2} immediately.
In Section \ref{sec4.1}, we show Theorem \ref{thm3.1}.
In Section \ref{sec4.2}, we prove Theorem \ref{thm3.2}.
\subsection{Proof of Theorem \ref{thm3.1}} \label{sec4.1}
In this section, we prove Theorem \ref{thm3.1}.
The method of the proof is the same one that is used in Section \ref{sec2}.
We use Theorem \ref{thm1.1} mainly.\\
~\\
{\bf Proof of Theorem \ref{thm3.1}.}
First, we treat $\mathbb{Q}(\sqrt{x_1^2 - p^n})$.
We can write 
$$(p)^n = (x_1 + \sqrt{x_1^2 - p^n})(x_1 - \sqrt{x_1^2 - p^n})$$
in $\mathbb{Q}(\sqrt{x_1^2 - p^n})$.
Put $\alpha := x_1 + \sqrt{x_1^2 - p^n}$.
Since $\gcd(p, x_1^2 - p^n) = 1$ and $p \nmid \alpha$ hold, $p$ splits in $\mathbb{Q}(\sqrt{x_1^2 - p^n})$.
Note that ideals $(\alpha)$ and $(\overline{\alpha})$ are coprime,
where $\overline{\alpha}$ is the complex conjugate of $\alpha$.
Then, 
$$(\alpha) = \wp^n,$$
where $\wp$ is the prime ideal of $\mathcal{O}_{\mathbb{Q}(\sqrt{x_1^2 - p^n})}$ over $p$.
Hence it follows from Theorem \ref{thm1.1} that $\lambda_p(\mathbb{Q}(\sqrt{x_1^2 - p^n})) = 1$
if and only if $p \nmid h(\mathbb{Q}(\sqrt{x_1^2 - p^n}))$ and
\begin{equation}
(x_1 + \sqrt{x_1^2 - p^n})^{p-1} - 1 \not \equiv 0 \bmod \overline{\wp}^2, \label{0}
\end{equation}
where $\overline{\wp}$ is the complex conjugate of $\wp$.
Equation (\ref{0}) is equivalent to
\begin{equation}
(x_1 + \sqrt{x_1^2 - p^n})^p - (x_1 + \sqrt{x_1^2 - p^n}) \not \equiv 0 \bmod \overline{\wp}^2\wp^n.  \label{1}
\end{equation}
We see that
\begin{equation}
 \begin{split}
 &(x_1 + \sqrt{x_1^2 - p^n})^p - (x_1 + \sqrt{x_1^2 - p^n}) \notag \\ 
= &\biggl\{\sum_{j=0}^p \binom{p}{j} x_1^{p-j}(\sqrt{x_1^2 - p^n})^j\biggr\} - (x_1 + \sqrt{x_1^2 - p^n}) \notag \\
= &\biggl\{\sum_{j=0}^{\frac{p-1}{2}} \binom{p}{2j} x_1^{p-2j}(x_1^2 - p^n)^j\biggr\} \notag \\
&+ \biggl\{\sum_{j=0}^{\frac{p-1}{2}}\binom{p}{2j+1}x_1^{p-2j-1}(x_1^2 - p^n)^j\sqrt{x_1^2 - p^n}\biggr\} - x_1 - \sqrt{x_1^2 - p^n}. \notag \\
 \end{split}
\end{equation}                       
When $n \ge 2$, we have $p^n \equiv 0 \bmod \overline{\wp}^2\wp^n$.
Then, 
\begin{equation}
 \begin{split}
 &(x_1 + \sqrt{x_1^2 - p^n})^p - (x_1 + \sqrt{x_1^2 - p^n}) \notag \\  
\equiv &\biggl\{\sum_{j=0}^{\frac{p-1}{2}} \binom{p}{2j} x_1^{p-2j}x_1^{2j}\biggr\}
 + \biggl\{\sum_{j=0}^{\frac{p-1}{2}}\binom{p}{2j+1}x_1^{p-2j-1}x_1^{2j}\sqrt{x_1^2 - p^n} \biggr\} \notag \\
 &- x_1 - \sqrt{x_1^2 - p^n} \notag \\
\equiv &\hspace{5pt}x_1^p \biggl\{\sum_{j=0}^{\frac{p-1}{2}} \binom{p}{2j}\biggr\}
 + x_1^{p-1}\sqrt{x_1^2 - p^n} \biggl\{\sum_{j=0}^{\frac{p-1}{2}}\binom{p}{2j+1}\biggr\} - x_1 - \sqrt{x_1^2 - p^n} \notag \\  
\equiv &\hspace{5pt}2^{p-1}x_1^p + 2^{p-1}x_1^{p-1}\sqrt{x_1^2 - p^n} - x_1 - \sqrt{x_1^2 - p^n} \notag \\
\equiv &\hspace{5pt}(x_1 + \sqrt{x_1^2 - p^n})(2^{p-1}x_1^{p-1} - 1) \bmod \overline{\wp}^2\wp^n. \notag \\
 \end{split}
\end{equation}
Equation (\ref{1}) is true if and only if
$$(2x_1)^{p-1} - 1 \not \equiv 0 \bmod p^2$$
by the above congruence relation.
Therefore, $\lambda_p(\mathbb{Q}(\sqrt{x_1^2 - p^n})) = 1$ if and only if $p \nmid h(\mathbb{Q}(\sqrt{x_1^2 - p^n}))$ and $(2x_1)^{p-1} \not \equiv 1 \bmod p^2$.
The proof of (1) is completed.\\
\hspace{17pt}Next, we treat $\mathbb{Q}(\sqrt{x_1^2 - 4p^n})$.
We can write 
$$(p)^n = \biggl(\frac{x_1 + \sqrt{x_1^2 - 4p^n}}{2}\biggr)\biggl(\frac{x_1 - \sqrt{x_1^2 - 4p^n}}{2}\biggr)$$
in $\mathbb{Q}(\sqrt{x_1^2 - 4p^n})$.
Put $\beta := \dfrac{x_1 + \sqrt{x_1^2 - 4p^n}}{2}$.
Since $\gcd(p, x_1^2 - 4p^n) = 1$ and $p \nmid \beta$ hold, $p$ splits in $\mathbb{Q}(\sqrt{x_1^2 - 4p^n})$.
Note that ideals $(\beta)$ and $(\overline{\beta})$ are coprime,
where $\overline{\beta}$ is the complex conjugate of $\beta$.
Then, 
$$(\beta) = \wp^n,$$
where $\wp$ is the prime ideal of $\mathcal{O}_{\mathbb{Q}(\sqrt{x_1^2 - 4p^n})}$ over $p$.
Hence it follows from Theorem \ref{thm1.1} that $\lambda_p(\mathbb{Q}(\sqrt{x_1^2 - 4p^n})) = 1$
if and only if $p \nmid h(\mathbb{Q}(\sqrt{x_1^2 - 4p^n}))$ and
\begin{equation}
\biggl(\frac{x_1 + \sqrt{x_1^2 - 4p^n}}{2}\biggr)^{p-1} - 1 \not \equiv 0 \bmod \overline{\wp}^2, \label{2}
\end{equation}
where $\overline{\wp}$ is the complex conjugate of $\wp$.
Equation (\ref{2}) is equivalent to
\begin{equation}
\biggl(\frac{x_1 + \sqrt{x_1^2 - 4p^n}}{2}\biggr)^p - \frac{x_1 + \sqrt{x_1^2 - 4p^n}}{2} \not \equiv 0 \bmod \overline{\wp}^2\wp^n.  \label{3}
\end{equation}
We see that
\begin{equation}
 \begin{split}
 &\biggl(\frac{x_1 + \sqrt{x_1^2 - 4p^n}}{2}\biggr)^p - \frac{x_1 + \sqrt{x_1^2 - 4p^n}}{2} \notag \\ 
= &\hspace{5pt}\frac{1}{2^p}(x_1 + \sqrt{x_1^2 - 4p^n})^p - \frac{1}{2}(x_1 + \sqrt{x_1^2 - 4p^n}) \notag \\ 
= &\hspace{5pt}\frac{1}{2^p}\biggl\{\sum_{j=0}^p \binom{p}{j} x_1^{p-j}(\sqrt{x_1^2 - 4p^n})^j\biggr\} - \frac{1}{2}(x_1 + \sqrt{x_1^2 - 4p^n}) \notag \\ 
= &\hspace{5pt}\frac{1}{2^p}\biggl\{\sum_{j=0}^{\frac{p-1}{2}} \binom{p}{2j} x_1^{p-2j}(x_1^2 - 4p^n)^j \notag \\
   &+ \sum_{j=0}^{\frac{p-1}{2}}\binom{p}{2j+1}x_1^{p-2j-1}(x_1^2 - 4p^n)^j\sqrt{x_1^2 - 4p^n}\biggr\} - \frac{1}{2}(x_1 + \sqrt{x_1^2 - 4p^n}). \notag \\ 
 \end{split}
\end{equation}                       
When $n \ge 2$, we have $p^n \equiv 0 \bmod \overline{\wp}^2\wp^n$.
Then, 
\begin{equation}
 \begin{split}
 &\biggl(\frac{x_1 + \sqrt{x_1^2 - 4p^n}}{2}\biggr)^p - \frac{x_1 + \sqrt{x_1^2 - 4p^n}}{2} \notag \\  
\equiv &\hspace{5pt}\frac{1}{2^p}\biggl\{\sum_{j=0}^{\frac{p-1}{2}} \binom{p}{2j} x_1^{p-2j}x_1^{2j} 
+ \sum_{j=0}^{\frac{p-1}{2}}\binom{p}{2j+1}x_1^{p-2j-1}x_1^{2j}\sqrt{x_1^2 - 4p^n}\biggr\} \notag \\ 
&- \frac{1}{2}(x_1 + \sqrt{x_1^2 - 4p^n}) \notag \\   
\equiv &\hspace{5pt}\frac{1}{2^p}\biggl\{x_1^p \sum_{j=0}^{\frac{p-1}{2}} \binom{p}{2j} + x_1^{p-1}\sqrt{x_1^2 - 4p^n} \sum_{j=0}^{\frac{p-1}{2}}\binom{p}{2j+1}\biggr\} - \frac{1}{2}(x_1 + \sqrt{x_1^2 - 4p^n}) \notag \\
\equiv &\hspace{5pt}\frac{1}{2^p}(2^{p-1}x_1^p + 2^{p-1}x_1^{p-1}\sqrt{x_1^2 - 4p^n}) - \frac{1}{2}(x_1 + \sqrt{x_1^2 - 4p^n}) \notag \\
\equiv &\hspace{5pt}\frac{1}{2}(x_1^p + x_1^{p-1}\sqrt{x_1^2 - 4p^n}) - \frac{1}{2}(x_1 + \sqrt{x_1^2 - 4p^n}) \notag \\
\equiv &\hspace{5pt}\frac{1}{2}(x_1 + \sqrt{x_1^2 - 4p^n})(x_1^{p-1} - 1) \bmod \overline{\wp}^2\wp^n. \notag \\
 \end{split}
\end{equation}                       
Equation (\ref{3}) is true if and only if
$$x_1^{p-1} - 1 \not \equiv 0 \bmod p^2$$
by the above congruence relation.
Therefore, $\lambda_p(\mathbb{Q}(\sqrt{x_1^2 - 4p^n})) = 1$ if and only if $p \nmid h(\mathbb{Q}(\sqrt{x_1^2 - 4p^n}))$ and $x_1^{p-1} \not \equiv 1 \bmod p^2$.
The proof of (2) is completed.
\subsection{Proof of Theorem \ref{thm3.2}} \label{sec4.2}
In this section, we show Theorem \ref{thm3.2} (1) and (2) (ii).
Theorem \ref{thm3.2} (2) (i) and (3) are already proved (see \cite{It11}, \cite{Is}, and \cite{It}).
First, we sketch the outline of the proof.
The method of the proof is based on the one in \cite{Ki} and \cite{It11}.

Let $p$ be an odd prime number and $n$ a positive integer satisfying the assumption in Theorem \ref{thm3.2} (1) and (2) (ii).
We treat the imaginary quadratic field $\mathbb{Q}(\sqrt{x_1^2 - p^n})$, where $x_1 = 1$, $q_1$, $2q_1$.
We can write 
$$(p)^n = (x_1 + \sqrt{x_1^2 - p^n})(x_1 - \sqrt{x_1^2 - p^n})$$
in $\mathbb{Q}(\sqrt{x_1^2 - p^n})$.
Put $\alpha := x_1 + \sqrt{x_1^2 - p^n}$.
Since $\gcd(p, x_1^2 - p^n) = 1$ and $p \nmid \alpha$ hold, the prime number $p$ splits in $\mathbb{Q}(\sqrt{x_1^2 - p^n})$.
Note that ideals $(\alpha)$ and $(\overline{\alpha})$ are coprime,
where $\overline{\alpha}$ is the complex conjugate of $\alpha$.
Then, 
$$(\alpha) = \wp^n,$$
where $\wp$ is the prime ideal of $\mathcal{O}_{\mathbb{Q}(\sqrt{x_1^2 - p^n})}$ over $p$.
We will prove the order of the ideal class containing $\wp$ is $n$.
It is essential for this proof to show that $\pm \alpha$ is not a $p_4$-th power in $\mathcal{O}_{\mathbb{Q}(\sqrt{x_1^2 - p^n})}$
 for any prime number $p_4$ dividing $n$.
In fact, we can prove this by checking that there is no integer solution $(u, v)$ of the equation 
$$\pm \alpha = (u + v\sqrt{d_0})^{p_4},$$
where $d_0$ is the square-free part of $x_1^2 - p^n$.

Next, we give a detailed proof of this. 
In Section \ref{sec4.2.1}, we prove Theorem \ref{thm3.2} (1) (i) and (ii). 
In Section \ref{sec4.2.2}, we show Theorem \ref{thm3.2} (1) (iii).
In Section \ref{sec4.2.3}, we prove Theorem \ref{thm3.2} (2) (ii).
\subsubsection{Proof of Theorem \ref{thm3.2} (1) (i) and (ii)} \label{sec4.2.1}  
In this section, we show Theorem \ref{thm3.2} (1) (i) and (ii).
First, we prepare the following lemma.
\begin{lem}\label{lem3.3}
Let $p$ be an odd prime number.
Then, for any given positive even integer $d_1$, the number of positive integer solutions $(x, y)$ of the equation 
$$d_1x^2 + 1 = p^y$$
is at most one except for $(d_1, p) = (2, 3)$.
When $(d_1, p) = (2, 3)$, the number of positive integer solutions $(x, y)$ of the above equation is three 
and the solutions are $(x, y) = (1, 1)$, $(2, 2)$, $(11, 5)$.
\end{lem}
We prove this by using a result of Y.~Bugeaud and T.~N.~Shorey~\cite{BS01}.
We state their result here.

We denote by $(\mathcal{F}_n)$ the Fibonacci sequence defined by
$\mathcal{F}_0 := 0$, $\mathcal{F}_1 := 1$, and satisfying $\mathcal{F}_{n+2} := \mathcal{F}_{n+1} + \mathcal{F}_{n}$ for all $n \ge 0$.
We denote by $(\mathcal{L}_n)$ the Lucas sequence defined by $\mathcal{L}_0 := 2$, $\mathcal{L}_1 := 1$,
 and satisfying $\mathcal{L}_{n+2} := \mathcal{L}_{n+1} + \mathcal{L}_{n}$ for all $n \ge 0$.
Then set
$$\mathfrak{F} := \{(\mathcal{F}_{h-2\varepsilon}, \mathcal{L}_{h+\varepsilon}, \mathcal{F}_h) \ | \ h \in \mathbb{N} \ s.t. \ h \ge 2 \ {\rm and} \ \varepsilon \in \{\pm 1\}\}$$
and
$$\mathfrak{G} := \{(1, 4p_5^{h_1} - 1, p_5) \ | \ p_5 \ {\rm is \ a \ prime \ number \ and} \ h_1 \in \mathbb{N}\}.$$
Bugeaud and Shorey gave the following result.
\begin{thm}[Bugeaud and Shorey, {[1, Theorem 1]}]\label{thm3.4} 
Let $d_1$ and $d_2$ be coprime positive integers and let $p$ be a prime number with $\gcd(d_1d_2, p) = 1$.
Let $\gamma \in \{1, \sqrt{2}, 2\}$ be such that $\gamma = 2$ if $p = 2$.
We assume that $d_2$ is odd if $\gamma \in \{\sqrt{2}, 2\}$.
Then, the number of positive integer solutions $(x, y)$ of the equation
$$d_1x^2 + d_2 = \gamma^2p^y$$
is at most one except for
\begin{eqnarray}
(\gamma, d_1, d_2, p) \in \mathfrak{E} := \left\{
\begin{array}{c}
(2, 13, 3, 2), (\sqrt{2}, 7, 11, 3), (1, 2, 1, 3), (2, 7, 1, 2), \nonumber \\
(\sqrt{2}, 1, 1, 5), (\sqrt{2}, 1, 1, 13), (2, 1, 3, 7). \nonumber \\
\end{array}
\right \}
\end{eqnarray}
or
$$(d_1, d_2, p) \in \mathfrak{F}\cup \mathfrak{G}\cup \mathfrak{H}_{\gamma},$$
where $\mathfrak{H}_{\gamma}$ denotes the set  
\begin{eqnarray}
\mathfrak{H}_{\gamma} := \left\{(d_1, d_2, p)\hspace{5pt}
\begin{array}{|c}
there \ exist \ positive \ integers \ s_0 \ and \ t_0 \nonumber \\
such \ that \ d_1s_0^2 + d_2 = \gamma^2p^{t_0} \ and \nonumber \\ 
3d_1s_0^2 - d_2 = \pm \gamma^2 \nonumber \\
\end{array}
\right \}.
\end{eqnarray}
\end{thm}
This theorem is also used in Sections \ref{sec4.2.2} and \ref{sec4.2.3}.\\
~\\  
{\bf Proof of Lemma \ref{lem3.3}.} \ We will show 
$$(\gamma, d_1, d_2, p) = (1, d_1, 1, p) \not \in \mathfrak{E} \smallsetminus \{(1, 2, 1, 3)\}$$
and
$$(d_1, d_2, p) = (d_1, 1, p) \not \in \mathfrak{F}\cup \mathfrak{G}\cup \mathfrak{H}_1$$
to use Theorem \ref{thm3.4} with $\gamma = 1$.
We see $(1, d_1, 1, p) \not \in \mathfrak{E} \smallsetminus \{(1, 2, 1, 3)\}$
and $(d_1, 1, p) \not \in \mathfrak{G}$ easily.
Suppose $(d_1, 1, p) \in \mathfrak{F}$.
There exists an integer $h$ greater than $1$ such that $\mathcal{F}_{h-2\varepsilon} = d_1$,
$\mathcal{L}_{h+\varepsilon} = 1$, and $\mathcal{F}_h = p$,
where $\varepsilon = \pm 1$.
The integer $h$ must be $0$ or $2$ by $\mathcal{L}_{h+\varepsilon} = 1$.
Since $h$ is greater than $1$, the integer $h$ must be $2$.
When $h = 2$, the $2$nd number in the Fibonacci sequence $\mathcal{F}_2$ is $1$.
This is a contradiction.
Then, $(d_1, 1, p) \not \in \mathfrak{F}$.
Next suppose $(d_1, 1, p) \in \mathfrak{H}_1$.
Both $d_1s_0^2 + 1 = p^{t_0}$ and $3d_1s_0^2 - 1 = \pm 1$ hold for some positive integers $s_0$ and $t_0$.
Then, we have $4 = 3p^{t_0} \pm 1$,
that is, $3p^{t_0} = 3$ or $5$.
This is a contradiction.
Therefore, $(d_1, 1, p) \not \in \mathfrak{H}_1$.
When $(d_1, p) = (2, 3)$, the statement of this lemma follows from \cite[Theorem 2.3]{LW03}.
The proof of Lemma \ref{lem3.3} is completed.
Next, we prove the following key lemma by using Lemma \ref{lem3.3}.
\begin{lem}\label{lem3.5}
Let $p$ be a prime number such that $p \equiv 3 \bmod 4$, $n$ an odd integer greater than $1$ such that $(p, n) \neq (3, 5)$, and 
$$\alpha := 1 + \sqrt{1 - p^n}.$$
Then, $\pm \alpha$ is not a $p_4$-th power in $\mathcal{O}_{\mathbb{Q}(\sqrt{1 - p^n})}$ for any prime $p_4$ dividing $n$.
\end{lem}
\begin{proof}
Since $p_4$ is odd, it is sufficient to prove that $\alpha$ is not a $p_4$-th power in $\mathcal{O}_{\mathbb{Q}(\sqrt{1 - p^n})}$.
We see from the congruence relations $p \equiv 3 \bmod 4$ and $n \equiv 1 \bmod 2$ that $1 - p^n \equiv 2 \bmod 4$.
Then $\mathcal{O}_{\mathbb{Q}(\sqrt{1 - p^n})} = \mathbb{Z}[\sqrt{d_0}]$, where $d_0$ is the square-free part of $1 - p^n$. 
Suppose $\alpha$ is a $p_4$-th power in $\mathcal{O}_{\mathbb{Q}(\sqrt{1 - p^n})}$.
We can write 
\begin{equation}
\alpha = (u + v\sqrt{d_0})^{p_4} \label{4}
\end{equation}
for some integers $u$ and $v$.
Expanding the right side of equation (\ref{4}), we have
$$1 + \sqrt{1 - p^n} = \biggl\{\sum_{j=0}^{\frac{p_4-1}{2}} \binom{p_4}{2j} u^{p_4-2j}(v^2d_0)^j\biggr\} + w\sqrt{d_0}$$
for some integer $w$.
Comparing the real parts of this equation,
$$1 = u \sum_{j=0}^{\frac{p_4-1}{2}} \binom{p_4}{2j} u^{p_4-2j-1}(v^2d_0)^j.$$
This implies that $u = \pm 1$.
Then, 
$$1 + \sqrt{1 - p^n} = (\pm 1 + v\sqrt{d_0})^{p_4}.$$
Taking the norm of this, we obtain $p^n = (1 - v^2d_0)^{p_4}$.
Then, 
$$p^{\frac{n}{p_4}} = 1 - v^2d_0.$$
On the other hand, we can write 
$$1 - p^n = c^2d_0$$
for some positive integer $c$.
These implies that both $(x, y) = (c, n)$ and $(x, y) = ( |v|, n / p_4)$ are positive integer solutions of the equation $-d_0x^2 + 1 = p^y$.
This is a contradiction when $(-d_0, p) \neq (2, 3)$ by Lemma \ref{lem3.3}.
When $(-d_0, p) = (2, 3)$, it follows from Lemma \ref{lem3.3} that $(c, n) = (1, 1)$, $(2, 2)$, $(11, 5)$.
Since $n$ is an odd integer greater than $1$, the pair of integers $(c, n)$ must be $(11, 5)$.
The case where $(p, n) = (3, 5)$ is not treated in this lemma.
The proof of Lemma \ref{lem3.5} is completed.
\end{proof}
Finally, we show Theorem \ref{thm3.2} (1) (i) and (ii) by using Lemma \ref{lem3.5}.\\
~\\
{\bf Proof of Theorem \ref{thm3.2} (1) (i) and (ii).} \ Since $\gcd(p, 1 - p^n) = 1$ and $p \nmid \alpha$ hold, the prime number $p$ splits in $\mathbb{Q}(\sqrt{1 - p^n})$.
Note that ideals $(\alpha)$ and $(\overline{\alpha})$ are coprime,
where $\overline{\alpha}$ is the complex conjugate of $\alpha$.
Then, 
$$(\alpha) = \wp^n,$$
where $\wp$ is the prime ideal of $\mathcal{O}_{\mathbb{Q}(\sqrt{1 - p^n})}$ over $p$.
Let $s_1$ be the order of the ideal class containing $\wp$.
We can write $n = s_1t_1$ for some integer $t_1$.
Since $\wp^{s_1}$ is principal, there exists some element $\eta \in \mathcal{O}_{\mathbb{Q}(\sqrt{1 - p^n})}$ such that $\wp^{s_1} = (\eta)$.
Then,
$$(\alpha) = (\wp^{s_1})^{t_1} = (\eta)^{t_1} = (\eta^{t_1}).$$
Since $\mathbb{Q}(\sqrt{1 - p^n}) \neq \mathbb{Q}(\sqrt{-1})$, $\mathbb{Q}(\sqrt{-3})$ holds, this implies that 
$$\pm \alpha = \eta^{t_1}.$$
When $(p, n) \neq (3, 5)$, we obtain $t_1 = 1$ by Lemma \ref{lem3.5}.
Then, $s_1 = n$ when $(p, n) \neq (3, 5)$.
The proof of Theorem \ref{thm3.2} (1) (i) and (ii) is completed.
\subsubsection{Proof of Theorem \ref{thm3.2} (1) (iii)} \label{sec4.2.2}
In this section, we show Theorem \ref{thm3.2} (1) (iii).
The method of the proof is basically similar to the one in Section \ref{sec4.2.1}.
First, we prepare the following two lemmas.
\begin{lem}\label{lem3.6}
Let $p$ and $q_5$ be distinct odd prime numbers.
Then, the equation 
$$4q_5^2 - 3p^x = \pm 1$$
has no positive integer solution $x$.
\end{lem}
\begin{proof}
When $q_5 = 3$, 
$$4q_5^2 - 3p^x \equiv 0 \bmod 3.$$
This is a contradiction.
Therefore, we treat the case where $q_5 \neq 3$.
First, we treat the equation $4q_5^2 - 3p^x = -1$.
Since $q_5^2 \equiv 1 \bmod 3$ holds, we have
$$4q_5^2 - 3p^x \equiv 4 - 0 \equiv 1 \bmod 3.$$
This is a contradiction.
Next, we treat the equation $4q_5^2 - 3p^x = 1$.
We can write 
$$3p^x = 4q_5^2 - 1 = (2q_5 + 1)(2q_5 - 1).$$
Since $\gcd(2q_5 +1, 2q_5 - 1) = 1$ holds, we obtain four cases:
(i) $2q_5 + 1 = p^x$, $2q_5 - 1 = 3$,
(ii) $2q_5 + 1 = 3$, $2q_5 - 1 = p^x$,
(iii) $2q_5 + 1 = 3p^x$, $2q_5 - 1 = 1$,
(iv) $2q_5 + 1 = 1$, $2q_5 - 1 = 3p^x$.
We easily see that the above four cases are impossible.
The proof of Lemma \ref{lem3.6} is completed.
\end{proof}
We use this to show the following lemma.
\begin{lem}\label{lem3.7}
Let $p$ and $q_5$ be distinct odd prime numbers.
Then, for any given positive even integer $d_1$, the number of positive integer solutions $(x, y)$ of the equation 
$$d_1x^2 + q_5^2 = p^y$$
is at most one.
\end{lem}
\begin{proof}
We will show 
$$(\gamma, d_1, d_2, p) = (1, d_1, q_5^2, p) \not \in \mathfrak{E}$$
and
$$(d_1, d_2, p) = (d_1, q_5^2, p) \not \in \mathfrak{F}\cup \mathfrak{G}\cup \mathfrak{H}_1$$
to use Theorem \ref{thm3.4} with $\gamma = 1$.
We see $(1, d_1, q_5^2, p) \not \in \mathfrak{E}$
and $(d_1, q_5^2, p) \not \in \mathfrak{G}$ easily.
Suppose $(d_1, q_5^2, p) \in \mathfrak{F}$.
There is an integer $h$ greater than $1$ such that $\mathcal{F}_{h-2\varepsilon} = d_1$,
$\mathcal{L}_{h+\varepsilon} = q_5^2$, and $\mathcal{F}_h = p$,
where $\varepsilon = \pm 1$.
J.~H.~E.~Cohn~\cite{Coh65} proved that $\mathcal{L}_1 = 1$ and $\mathcal{L}_3 = 4$ are the only squares in the Lucas sequence.
Then, $\mathcal{L}_{h+\varepsilon} = q_5^2$ is impossible.
Therefore, $(d_1, q_5^2, p) \not \in \mathfrak{F}$.
Next suppose $(d_1, q_5^2, p) \in \mathfrak{H}_1$.
Both $d_1s_0^2 + q_5^2 = p^{t_0}$ and $3d_1s_0^2 - q_5^2 = \pm 1$ hold for some positive integers $s_0$ and $t_0$.
Then, we have $4q_5^2 = 3p^{t_0} \pm 1$.
This is a contradiction with Lemma \ref{lem3.6}.
Therefore, $(d_1, q_5^2, p) \not \in \mathfrak{H}_1$.
The proof of Lemma \ref{lem3.7} is completed.
\end{proof}
Next, we prove the following key lemma by using Lemma \ref{lem3.7}.
\begin{lem}\label{lem3.8}
Let $p$ be a prime number such that $p \equiv 3 \bmod 4$ and $2^{p-1} \equiv 1 \bmod p^2$, 
$q_1$ a prime factor of $p - 2$ such that $q_1^{p-1} \not \equiv 1 \bmod p^2$,
$n$ an odd composite number, and 
$$\alpha := q_1 + \sqrt{q_1^2 - p^n}.$$
Then, $\pm \alpha$ is not a $p_4$-th power in $\mathcal{O}_{\mathbb{Q}(\sqrt{q_1^2 - p^n})}$ for any prime $p_4$ dividing $n$.
\end{lem}
\begin{proof}
Since $p_4$ is odd, it is sufficient to prove that $\alpha$ is not a $p_4$-th power in $\mathcal{O}_{\mathbb{Q}(\sqrt{q_1^2 - p^n})}$.
Suppose $\alpha$ is a $p_4$-th power in $\mathcal{O}_{\mathbb{Q}(\sqrt{q_1^2 - p^n})}$.
It follows from the assumptions $p \equiv 3 \bmod 4$, $q_1 \equiv 1 \bmod 2$ and $n \equiv 1 \bmod 2$ that $q_1^2 - p^n \equiv 2 \bmod 4$.
Then $\mathcal{O}_{\mathbb{Q}(\sqrt{q_1^2 - p^n})} = \mathbb{Z}[\sqrt{d_0}]$, where $d_0$ is the square-free part of $q_1^2 - p^n$. 
We can write 
\begin{equation}
\alpha = (u + v\sqrt{d_0})^{p_4} \label{5}
\end{equation}
for some integers $u$ and $v$.
Expanding the right side of equation (\ref{5}), we have
$$q_1 + \sqrt{q_1^2 - p^n} = \biggl\{\sum_{j=0}^{\frac{p_4-1}{2}} \binom{p_4}{2j} u^{p_4-2j}(v^2d_0)^j\biggr\} + w\sqrt{d_0}$$
for some integer $w$.
Comparing the real parts of this equation,
\begin{equation}
q_1 = u \sum_{j=0}^{\frac{p_4-1}{2}} \binom{p_4}{2j} u^{p_4-2j-1}(v^2d_0)^j. \label{6}
\end{equation}
This implies that $u = \pm 1$, $\pm q_1$.
Suppose $u = \pm q_1$.
Then, 
$$q_1 + \sqrt{q_1^2 - p^n} = (\pm q_1 + v\sqrt{d_0})^{p_4}.$$
Taking the norm of this, we obtain $p^n = (q_1^2 - v^2d_0)^{p_4}$.
Then, 
$$p^{\frac{n}{p_4}} = q_1^2 - v^2d_0.$$
On the other hand, we can write 
$$q_1^2 - p^n = c^2d_0$$
for some positive integer $c$.
These implies that both $(x, y) = (c, n)$ and $(x, y) = ( |v|, n / p_4)$ are positive integer solutions of the equation $-d_0x^2 + q_1^2 = p^y$.
This is a contradiction by Lemma \ref{lem3.7}.
Then, $u$ must be $\pm 1$.
Substituting $u = \pm 1$ in equation (\ref{6}), we obtain
\begin{equation}
\pm q_1 = \sum_{j=0}^{\frac{p_4-1}{2}} \binom{p_4}{2j} (v^2d_0)^j. \label{7}
\end{equation}
Taking the norm of $\alpha = (\pm 1 + v\sqrt{d_0})^{p_4}$, we have $p^n = (1 - v^2d_0)^{p_4}$.
Then, $p^{\frac{n}{p_4}} = 1 - v^2d_0$, that is, $v^2d_0 = 1 - p^{\frac{n}{p_4}}$.
Substituting this in equation (\ref{7}), we obtain
$$\pm q_1 = \sum_{j=0}^{\frac{p_4-1}{2}} \binom{p_4}{2j} (1 - p^{\frac{n}{p_4}})^j.$$
Since $n$ is a composite number, we have $n / p_4 \ge 2$.
Then, $p^{n / p_4} \equiv 0 \bmod p^2$.
Using this, we see
$$\sum_{j=0}^{\frac{p_4-1}{2}} \binom{p_4}{2j} (1 - p^{\frac{n}{p_4}})^j \equiv \sum_{j=0}^{\frac{p_4-1}{2}} \binom{p_4}{2j} \equiv 2^{p_4 - 1} \bmod p^2.$$
Then,
$$\pm q_1 \equiv 2^{p_4 - 1} \bmod p^2.$$
Combining this and the assumption $2^{p-1} \equiv 1 \bmod p^2$, we have
$$q_1^{p-1} \equiv (\pm q_1)^{p-1} \equiv (2^{p_4 - 1})^{p-1} \equiv (2^{p-1})^{p_4 - 1} \equiv 1 \bmod p^2.$$
This is a contradiction.
The proof of Lemma \ref{lem3.8} is completed.
\end{proof}
Finally, we show Theorem \ref{thm3.2} (1) (iii) by using Lemma \ref{lem3.8}.\\
~\\
{\bf Proof of Theorem \ref{thm3.2} (1) (iii).} \ Since $\gcd(p, q_1^2 - p^n) = 1$ and $p \nmid \alpha$ hold, the prime number $p$ splits in $\mathbb{Q}(\sqrt{q_1^2 - p^n})$.
Note that ideals $(\alpha)$ and $(\overline{\alpha})$ are coprime,
where $\overline{\alpha}$ is the complex conjugate of $\alpha$.
Then, 
$$(\alpha) = \wp^n,$$
where $\wp$ is the prime ideal of $\mathcal{O}_{\mathbb{Q}(\sqrt{q_1^2 - p^n})}$ over $p$.
Let $s_1$ be the order of the ideal class containing $\wp$.
We can write $n = s_1t_1$ for some integer $t_1$.
Since $\wp^{s_1}$ is principal, there exists some element $\eta \in \mathcal{O}_{\mathbb{Q}(\sqrt{q_1^2 - p^n})}$ such that $\wp^{s_1} = (\eta)$.
Then,
$$(\alpha) = (\wp^{s_1})^{t_1} = (\eta)^{t_1} = (\eta^{t_1}).$$
Since $\mathbb{Q}(\sqrt{q_1^2 - p^n}) \neq \mathbb{Q}(\sqrt{-1})$, $\mathbb{Q}(\sqrt{-3})$ holds, this implies that 
$$\pm \alpha = \eta^{t_1}.$$
We obtain $t_1 = 1$ by Lemma \ref{lem3.8}.
Then, $s_1 = n$.
The proof of Theorem \ref{thm3.2} (1) (iii) is completed.
\subsubsection{Proof of Theorem 4.2 (2) (ii)} \label{sec4.2.3}
In this section, we show Theorem 4.2 (2) (ii).
The method of the proof is also basically similar to the one in Section \ref{sec4.2.1}.
First, we prepare the following two lemmas.
\begin{lem}\label{lem3.9}
Let $p$ and $q_6$ be distinct odd prime numbers.
Then, the equation 
$$16q_6^2 - 3p^x = \pm 1$$
has no positive integer solution $x$.
\end{lem}
\begin{proof}
When $q_6 = 3$, 
$$16q_6^2 - 3p^x \equiv 0 \bmod 3.$$
This is a contradiction.
Therefore, we treat the case where $q_6 \neq 3$.
First, we treat the equation $16q_6^2 - 3p^x = -1$.
Since $q_6^2 \equiv 1 \bmod 3$ holds, we have
$$16q_6^2 - 3p^x \equiv 1 - 0 \equiv 1 \bmod 3.$$
This is a contradiction.
Next, we treat the equation $16q_6^2 - 3p^x = 1$.
We can write 
$$3p^x = 16q_6^2 - 1 = (4q_6 + 1)(4q_6 - 1).$$
Since $\gcd(4q_6 +1, 4q_6 - 1) = 1$ holds, we obtain four cases:
(i) $4q_6 + 1 = p^x$, $4q_6 - 1 = 3$,
(ii) $4q_6 + 1 = 3$, $4q_6 - 1 = p^x$,
(iii) $4q_6 + 1 = 3p^x$, $4q_6 - 1 = 1$,
(iv) $4q_6 + 1 = 1$, $4q_6 - 1 = 3p^x$.
We easily see that the above four cases are impossible.
The proof of Lemma \ref{lem3.9} is completed.
\end{proof}
We use this to show the following lemma.
\begin{lem}\label{lem3.10}
Let $p$ and $q_6$ be distinct odd prime numbers.
Then, for any given positive odd integer $d_1$, the number of positive integer solutions $(x, y)$ of the equation 
$$d_1x^2 + 4q_6^2 = p^y$$
is at most one.
\end{lem}
\begin{proof}
We will show 
$$(\gamma, d_1, d_2, p) = (1, d_1, 4q_6^2, p) \not \in \mathfrak{E}$$
and
$$(d_1, d_2, p) = (d_1, 4q_6^2, p) \not \in \mathfrak{F}\cup \mathfrak{G}\cup \mathfrak{H}_1$$
to use Theorem \ref{thm3.4} with $\gamma = 1$.
We see $(1, d_1, 4q_6^2, p) \not \in \mathfrak{E}$
and $(d_1, 4q_6^2, p) \not \in \mathfrak{G}$ easily.
Suppose $(d_1, 4q_6^2, p) \in \mathfrak{F}$.
There exists an integer $h$ greater than $1$ such that $\mathcal{F}_{h-2\varepsilon} = d_1$,
$\mathcal{L}_{h+\varepsilon} = 4q_6^2$, and $\mathcal{F}_h = p$,
where $\varepsilon = \pm 1$.
For any integer $h \ge 2$, we find
$$4\mathcal{F}_h - \mathcal{F}_{h-2\varepsilon} = \mathcal{L}_{h+\varepsilon}$$
(see \cite[Lemma 3]{BS01}).
Using this, we have
$$4p - d_1 = 4q_6^2.$$
Since $4p - d_1$ is odd, this is impossible.
Therefore, $(d_1, 4q_6^2, p) \not \in \mathfrak{F}$.
Next suppose $(d_1, 4q_6^2, p) \in \mathfrak{H}_1$.
Both $d_1s_0^2 + 4q_6^2 = p^{t_0}$ and $3d_1s_0^2 - 4q_6^2 = \pm 1$ hold for some positive integers $s_0$ and $t_0$.
Then, we have $16q_6^2 = 3p^{t_0} \pm 1$.
This is a contradiction with Lemma \ref{lem3.9}.
Therefore, $(d_1, 4q_6^2, p) \not \in \mathfrak{H}_1$.
The proof of Lemma \ref{lem3.10} is completed.
\end{proof}
Next, we prove the following key lemma by using Lemma \ref{lem3.10}.
\begin{lem}\label{lem3.11}
Let $p$ be a prime number such that $p \equiv 1 \bmod 4$ and $2^{p-1} \equiv 1 \bmod p^2$, 
$q_1$ a prime factor of $p - 2$ such that $q_1^{p-1} \not \equiv 1 \bmod p^2$,
$n$ a composite number, and 
$$\alpha := 2q_1 + \sqrt{4q_1^2 - p^n}.$$
Then, $\pm \alpha$ is not a $p_4$-th power in $\mathcal{O}_{\mathbb{Q}(\sqrt{4q_1^2 - p^n})}$ for any prime $p_4$ dividing $n$.
\end{lem}
\begin{proof}
Suppose $\pm \alpha$ is a $p_4$-th power in $\mathcal{O}_{\mathbb{Q}(\sqrt{4q_1^2 - p^n})}$.
Since $p \equiv 1 \bmod 4$ holds, we see $4q_1^2 - p^n \equiv 3 \bmod 4$.
Then $\mathcal{O}_{\mathbb{Q}(\sqrt{4q_1^2 - p^n})} = \mathbb{Z}[\sqrt{d_0}]$, where $d_0$ is the square-free part of $4q_1^2 - p^n$. 
We can write 
\begin{equation}
\pm \alpha = (u + v\sqrt{d_0})^{p_4} \label{8}
\end{equation}
for some integers $u$ and $v$.
First, we treat the case where $p_4 = 2$.
Substituting $p_4 = 2$ in equation (\ref{8}), we have
$$\pm \alpha = (u + v\sqrt{d_0})^2 = (u^2 + v^2d_0) + 2uv\sqrt{d_0}.$$
On the other hand, we can write 
$$4q_1^2 - p^n = c^2d_0$$
for some positive odd integer $c$.
Using this expression, we have 
$$\pm \alpha = \pm 2q_1 \pm c\sqrt{d_0}.$$
Comparing the imaginary parts of these equation,
$$\pm c = 2uv.$$
Since $c$ is odd, this is impossible.
Then, $\pm \alpha$ is not a square in $\mathcal{O}_{\mathbb{Q}(\sqrt{4q_1^2 - p^n})}$.
Next, we treat the case where $p_4 \ge 3$.
It is sufficient to prove that $\alpha$ is not a $p_4$-th power in $\mathcal{O}_{\mathbb{Q}(\sqrt{4q_1^2 - p^n})}$.
Expanding the right side of the equation $\alpha = (u + v\sqrt{d_0})^{p_4}$, we have
\begin{equation}
 \begin{split}
2q_1 + \sqrt{4q_1^2 - p^n} &= \biggl\{\sum_{j=0}^{\frac{p_4-1}{2}} \binom{p_4}{2j} u^{p_4-2j}(v^2d_0)^j\biggr\} + w\sqrt{d_0} \notag \\
                                                 &= u \biggl\{\sum_{j=0}^{\frac{p_4-1}{2}} \binom{p_4}{2j} u^{p_4-2j-1}(v^2d_0)^j\biggr\} + w\sqrt{d_0} \notag \\                         
 \end{split}
\end{equation}                       
for some integer $w$.
Comparing the real parts of this equation,
\begin{equation}
2q_1 = u \sum_{j=0}^{\frac{p_4-1}{2}} \binom{p_4}{2j} u^{p_4-2j-1}(v^2d_0)^j. \label{9}
\end{equation}
Then, $u = \pm 1$, $\pm 2$, $\pm q_1$, $\pm 2q_1$.
Suppose $u = \pm 1$, $\pm 2$.
Taking the norm of $\alpha = (u + v\sqrt{d_0})^{p_4}$, we have $p^n = (u^2 - v^2d_0)^{p_4}$.
Then, $p^{n / p_4} = u^2 - v^2d_0$,
that is, 
$$v^2d_0 = u^2 - p^{\frac{n}{p_4}}.$$
Substituting this in equation (\ref{9}), we obtain
$$2q_1 = u\sum_{j=0}^{\frac{p_4-1}{2}} \binom{p_4}{2j} u^{p_4-2j-1}(u^2 - p^{\frac{n}{p_4}})^j.$$
Since $n$ is a composite number, we have $n / p_4 \ge 2$.
Then, $p^{n / p_4} \equiv 0 \bmod p^2$.
Using this, we see
\begin{equation}
 \begin{split}
u\sum_{j=0}^{\frac{p_4-1}{2}} \binom{p_4}{2j} u^{p_4-2j-1}(u^2 - p^{\frac{n}{p_4}})^j &\equiv \sum_{j=0}^{\frac{p_4-1}{2}} \binom{p_4}{2j} u^{p_4-2j}u^{2j} \notag \\
                                                                                          &\equiv u^{p_4} \sum_{j=0}^{\frac{p_4-1}{2}} \binom{p_4}{2j} \equiv 2^{p_4-1}u^{p_4} \bmod p^2. \notag \\
 \end{split}
\end{equation}                       
Since $p$ is odd, we have
$$q_1 \equiv 2^{p_4-2}u^{p_4} \bmod p^2.$$
Combining this and the assumption $2^{p-1} \equiv 1 \bmod p^2$, we see
\begin{equation}
 \begin{split}
q_1^{p-1} \equiv (2^{p_4-2}u^{p_4})^{p-1} &\equiv (2^{p_4 - 2})^{p-1}(u^{p_4})^{p-1} \notag \\ 
                   &\equiv (2^{p-1})^{p_4 - 2}(u^{p-1})^{p_4} \equiv (u^{p-1})^{p_4} \bmod p^2. \notag \\
 \end{split}
\end{equation}                       
When $u = \pm 1$, we have $u^{p-1} = 1$.
When $u = \pm 2$, we have $u^{p-1} = 2^{p-1} \equiv 1 \bmod p^2$.
Then,
$$q_1^{p-1} \equiv (u^{p-1})^{p_4} \equiv 1 \bmod p^2.$$
This is a contradiction.
Therefore, it must be $u = \pm q_1$, $\pm 2q_1$.
Suppose $u = \pm q_1$.
Substituting this in equation (\ref{9}), we obtain
$$2q_1 = \pm q_1 \sum_{j=0}^{\frac{p_4-1}{2}} \binom{p_4}{2j} q_1^{p_4-2j-1}(v^2d_0)^j.$$
Then,
\begin{equation}
\pm 2 = \sum_{j=0}^{\frac{p_4-1}{2}} \binom{p_4}{2j} q_1^{p_4-2j-1}(v^2d_0)^j = q_1^{p_4-1} + \sum_{j=1}^{\frac{p_4-1}{2}} \binom{p_4}{2j} q_1^{p_4-2j-1}(v^2d_0)^j. \label{10}
\end{equation}                       
Taking the norm of $\alpha = (\pm q_1 + v\sqrt{d_0})^{p_4}$, we have 
$$v^2d_0 = q_1^2 - p^{\frac{n}{p_4}}.$$
Since $q_1^2 - p^{n / p_4}$ is even and $d_0$ is odd, we see that $v$ is even.
Then, the right side of equation (\ref{10}) is odd.
This is a contradiction.
Therefore, it must be $u = \pm 2q_1$.
Taking the norm of $\alpha = (\pm 2q_1 + v\sqrt{d_0})^{p_4}$, we obtain
$$v^2d_0 = 4q_1^2 - p^{\frac{n}{p_4}}.$$
On the other hand,
$$4q_1^2 - p^n = c^2d_0.$$
These implies that both $(x, y) = (c, n)$ and $(x, y) = ( |v|, n / p_4)$ are positive integer solutions of the equation $-d_0x^2 + 4q_1^2 = p^y$.
This is a contradiction by Lemma \ref{lem3.10}.
The proof of Lemma \ref{lem3.11} is completed.
\end{proof}
Finally, we show Theorem \ref{thm3.2} (2) (ii) by using Lemma \ref{lem3.11}.\\
~\\
{\bf Proof of Theorem \ref{thm3.2} (2) (ii).} \ Since $\gcd(p, 4q_1^2 - p^n) = 1$ and $p \nmid \alpha$ hold, 
the prime number $p$ splits in $\mathbb{Q}(\sqrt{4q_1^2 - p^n})$.
Note that ideals $(\alpha)$ and $(\overline{\alpha})$ are coprime,
where $\overline{\alpha}$ is the complex conjugate of $\alpha$.
Then, 
$$(\alpha) = \wp^n,$$
where $\wp$ is the prime ideal of $\mathcal{O}_{\mathbb{Q}(\sqrt{4q_1^2 - p^n})}$ over $p$.
Let $s_1$ be the order of the ideal class containing $\wp$.
We can write $n = s_1t_1$ for some integer $t_1$.
Since $\wp^{s_1}$ is principal, there exists some element $\eta \in \mathcal{O}_{\mathbb{Q}(\sqrt{4q_1^2 - p^n})}$ such that $\wp^{s_1} = (\eta)$.
Then,
$$(\alpha) = (\wp^{s_1})^{t_1} = (\eta)^{t_1} = (\eta^{t_1}).$$
Since $d_0 \equiv -1 \bmod 4$ holds, $\mathbb{Q}(\sqrt{4q_1^2 - p^n}) \neq \mathbb{Q}(\sqrt{-3})$.
From the assumption $\mathbb{Q}(\sqrt{4q_1^2 - p^n}) \neq \mathbb{Q}(\sqrt{-1})$, we have
$$\mathbb{Q}(\sqrt{4q_1^2 - p^n}) \neq \mathbb{Q}(\sqrt{-1}), \ \mathbb{Q}(\sqrt{-3}).$$
Then,
$$\pm \alpha = \eta^{t_1}.$$
We obtain $t_1 = 1$ by Lemma \ref{lem3.11}.
Therefore, $s_1 = n$.
The proof of Theorem \ref{thm3.2} (2) (ii) is completed.
\begin{rem}
We give a remark on Theorem \ref{thm3.2} (2).
We see how often the equality 
$\mathbb{Q}(\sqrt{4 - p^n}) = \mathbb{Q}(\sqrt{-1})$ (resp. $\mathbb{Q}(\sqrt{4q_1^2 - p^n}) = \mathbb{Q}(\sqrt{-1})$)
occurs as $n$ runs over positive integers for a fixed prime number $p$ (resp. for fixed prime numbers $p$ and $q_1$).
The number of positive integer solutions $(x, y)$ of the equation
$$x^2 + 4 = p^y$$
is at most one except for $p = 5$ (see \cite[Lemma 3]{It11}).
Then, for a fixed $p$, the number of positive integers $n$ such that $\mathbb{Q}(\sqrt{4 - p^n}) = \mathbb{Q}(\sqrt{-1})$ is at most one except for $p = 5$.
By Lemma \ref{lem3.10} of this paper, 
the number of positive integer solutions $(x, y)$ of the equation 
$$x^2 + 4q_1^2 = p^y$$
is at most one.
Then, for fixed odd prime numbers $p$ and $q_1$, the number of positive integers $n$ such that $\mathbb{Q}(\sqrt{4q_1^2 - p^n}) = \mathbb{Q}(\sqrt{-1})$ is at most one.
\end{rem}
\subsection*{Acknowledgements}
The author would like to express sincere gratitude to Professor Masato Kurihara and Doctor Satoshi Fujii for suggesting this subject and for valuable advice.
She would like to thank Professor Manabu Ozaki for answering her questions.
She would also like to thank Professor Hisao Taya and Doctor Takayuki Morisawa for helpful discussion.
Further, she would like to express my thanks 
to Professor Akio Tamagawa, Professor Yasushi Mizusawa, Professor Tsuyoshi Itoh, Professor Dongho Byeon, and 
Doctor Filippo Alberto Edoardo Nuccio Mortarino Majno di Capriglio for useful comments.
Finally, she wishes to be grateful to Professor Kohji Matsumoto for continuous encouragement.

~\\
Graduate School of Mathematics\\
Nagoya University\\
Chikusa-ku, Nagoya City\\
Aichi 464-8602, Japan\\
E-mail: m07004a@math.nagoya-u.ac.jp

\begin{thebibliography}{9}
\bibitem{BS01} Y. Bugeaud and T. N. Shorey,
\textit{On the number of solutions of the generalized Ramanujan-Nagell equation},
J. Reine Angew. Math. {\bf 539} (2001), 55--74.
\bibitem{By99} D. Byeon,
\textit{A note on basic Iwasawa $\lambda$-invariants of imaginary quadratic fields and congruence of modular forms},
Acta Arith. {\bf 89} (1999), no.~3, 295--299. 
\bibitem{By01} D. Byeon,
\textit{Indivisibility of class numbers and Iwasawa $\lambda$-invariants of real quadratic fields},
Compositio Math. {\bf 126} (2001), no.~3, 249--256.
\bibitem{By02} D. Byeon,
\textit{A note on the existence of certain infinite families of imaginary quadratic fields},
J. Number Theory {\bf 97} (2002), no.~1, 165--170.
\bibitem{By05} D. Byeon,
\textit{Imaginary quadratic fields whose Iwasawa $\lambda$-invariant is equal to $1$},
Acta Arith. {\bf 120} (2005), no.~2, 145--152.
\bibitem{Ca59} L. Carlitz,
\textit{Arithmetic properties of generalized Bernoulli numbers},
J. Reine. Angew. Math. {\bf 202} (1959), 174--182.
\bibitem{Co75} H. Cohen,
\textit{Sums involving the values at negative integers of $L$-functions of quadratic characters},
Math. Ann. {\bf 217} (1975), 271--285.
\bibitem{Co} H. Cohen,
\textit{Number Theory Volume II : Analytic and Modern Tools},
Grad. Texts in Math., vol. {\bf 240}, Springer.
\bibitem{Coh65} J. H. E. Cohn,
\textit{Lucas and Fibonacci numbers and some Diophantine equations},
Proc. Glasgow Math. Assoc. {\bf 7} (1965), 24--28. 
\bibitem{DFKS91} D. S. Dummit, D. Ford, H. Kisilevsky and J. W. Sands,
\textit{Computation of Iwasawa lambda invariants for imaginary quadratic fields},
J. Number Theory {\bf 37} (1991), 100--121.
\bibitem{FT} T. Fukuda and H. Taya,
\textit{Computation of Iwasawa Invariants}, 
Nihon Oyo Surigakkai Ronbunshi, {\bf 12}, (2002), no.~4, 293--306.
\bibitem{Go74} R. Gold,
\textit{The nontriviality of certain $\mathbb{Z}_l$-extensions},
J. Number Theory {\bf 6} (1974), 369--373.
\bibitem{Gr73} R. Greenberg,
\textit{Iwasawa theory --- past and present},
Adv. Stud. Pure Math. {\bf 30}, Math. Soc. Japan, Tokyo, (2001),  335--386.
\bibitem{Is} K. Ishii,
\textit{On the divisibility of the class number of imaginary quadratic fields},
Proc. Japan Acad. Ser. A Math. Sci. {\bf 87} (2011), 142--143. 
\bibitem{It11} A. Ito,
\textit{Remarks on the divisibility of class numbers of imaginary quadratic fields $\mathbb{Q}(\sqrt{2^{2k} - q^n})$},
Glasgow Math. J. {\bf 53} (2011), 379--389. 
\bibitem{It} A. Ito,
\textit{Notes on the divisibility of the class numbers of imaginary quadratic fields $\mathbb{Q}(\sqrt{3^{2e} - 4k^n})$},
preprint. 
\bibitem{JO99}K. James and K. Ono,
\textit{Selmer groups of quadratic twists of elliptic curves},
Math. Ann. {\bf 314} (1999), 1--17. 
\bibitem{Kid79} Y. Kida,
\textit{On cyclotomic $\mathbb{Z}_2$-extensions of imaginary quadratic fields},
T\^{o}hoku Math. J. (2) {\bf 31} (1979), no.~1, 91--96.
\bibitem{Kim03} I. Kimura,
\textit{A note on the existence of certain infinite families of imaginary quadratic fields},
Acta Arith. {\bf 110} (2003), no.~1, 37--43; Corrigendum, Acta Arith. {\bf 114} (2004), no.~4, 397.
\bibitem{Ki} Y. Kishi, 
\textit{Note on the divisibility of the class number of certain imaginary quadratic fields}, 
Glasgow Math. J. {\bf 51} (2009), 187--191; Corrigendum, Glasgow Math. J. {\bf 52} (2010), no.~1, 207--208.
\bibitem{KO99} W. Kohnen and K. Ono,
\textit{Indivisibility of class numbers of imaginary quadratic fields and orders of Tate-Shafarevich groups of elliptic curves with complex multiplication},
Invent. Math. {\bf 135} (1999), no.~2, 387--398.
\bibitem{KW07} J. S. Kraft and L. C. Washington,
\textit{Heuristics for class numbers and lambda invariants},
Math. Comp. {\bf 76} (2007), no.~258, 1005--1023.
\bibitem{LW03} M.-G. Leu and G.-W. Li,
\textit{The Diophantine equation $2x^2 + 1 = 3^n$},
Proc. Amer. Math. Soc. {\bf 131} (2003), no.~12, 3643--3645. 
\bibitem{Mi}Y. Mizusawa, 
http://mizusawa.web.nitech.ac.jp/Iwapoly/index.html. 
\bibitem{On99}K. Ono, 
\textit{Indivisibility of class numbers of real quadratic fields}, 
Compositio Math. {\bf 119} (1999), no.~1, 1--11.
\bibitem{Sa93}J. W. Sands, 
\textit{On the nontriviality of the basic Iwasawa $\lambda$-invariant for an infinitude of imaginary quadratic fields},
Acta Arith. {\bf 65} (1993), no.~3, 243--248.
\bibitem{St87}J. Sturm, 
\textit{On the congruence of modular forms}, 
in: Number Theory (New York, 1984-1985), 
Lecture Notes in Math. {\bf 1240}, Springer, Berlin, 1987, 275--280.
\bibitem{Wa82}L. C. Washington, 
\textit{Zeros of $p$-adic $L$-functions}, 
in: S\'{e}minaire Delange-Pisot-Poitou(S\'{e}minaire de Th\'{e}orie des Nombres, Paris, 1980/1981), 
Progr. Math. {\bf 22}, Birkh\"{a}user, Boston, 1982, 337--357.
\bibitem{Wa}L. C. Washington, 
\textit{Introduction to Cyclotomic fields}, 
second edition, GTM {\bf 83}, Springer.
\end{thebibliography}
\end{document}